\documentclass[11pt]{article}
\usepackage[shortlabels]{enumitem}
\usepackage{algorithm}
\usepackage{algorithmic}
\usepackage{stmaryrd}

\usepackage{amsmath,amsfonts,amsthm,amsbsy,amssymb}
\usepackage{latexsym,mathrsfs}
\usepackage{color}

\usepackage{cleveref}
\crefname{assumption}{assumption}{assumptions}

\usepackage{graphicx}
\graphicspath{{./}}

\def\bbC{\mathbb{C}}

\def\bbO{\mathbb{O}}
\def\bbR{\mathbb{R}}
\def\bbU{\mathbb{U}}

\def\SL{S} 
\def\LOP{local $\bbR$-linear operator}
\def\QOP{restricted derivative operator}

\def\OP{\mathscr{L}} 

\def\lo{\mathbf{o}}
\def\calo{\mathcal{O}}

\def\bbR{\mathbb{R}}

\def\scrQ{\mathscr{Q}}

\def\cR{{\cal R}}

\usepackage{mathabx}  

\def\what{\widehat}

\DeclareMathOperator{\czbl}{czbl}
\DeclareMathOperator{\diag}{diag}

\DeclareMathOperator{\tr}{tr}
\DeclareMathOperator{\tridiag}{tridiag}
\DeclareMathOperator{\UI}{ui}

\DeclareMathOperator{\Diag}{Diag}
\DeclareMathOperator{\F}{F}
\DeclareMathOperator{\HH}{H}

\DeclareMathOperator{\T}{T}

\newtheorem{assumption}{Assumption}

\newtheorem{theorem}{Theorem}
\newtheorem{lemma}{Lemma}

\theoremstyle{definition}

\newtheorem{example}{Example}

\numberwithin{equation}{section}

\newcommand{\twobytwo}[4]{
       \left[ \begin{array}{cc}
        #1 & #2  \\
        #3 & #4
           \end{array} \right] }

\newcommand{\ignore}[1]{}

\title{
Optimal Convergence Rate of Self-Consistent Field Iteration
for Solving Eigenvector-dependent Nonlinear Eigenvalue Problems
}

\author{
    Zhaojun Bai\thanks{Department of Computer Science, University of California, Davis, CA 95616, USA, {\tt zbai@ucdavis.edu}}
    \and Ren-Cang Li\thanks{ Department of Mathematics, University of Texas at Arlington, Arlington, TX 76019-0408, USA, {\tt rcli@uta.edu}}
    \and Ding Lu\thanks{Department of Mathematics, University of Kentucky, Lexington, KY 40506, USA, {\tt Ding.Lu@uky.edu}} 
}

\date{\today}

\begin{document}

\maketitle

\begin{abstract}
We present a comprehensive convergence analysis for Self-Consistent
Field (SCF) iteration to solve a class of nonlinear
eigenvalue problems with eigenvector-dependency (NEPv).
Using a tangent-angle matrix as an intermediate measure
for approximation error,
we establish new formulas for two fundamental quantities
that optimally characterize the local convergence of the plain SCF:
the local contraction factor and the local average contraction factor.
In comparison with previously established results,
new convergence rate estimates provide much sharper bounds on
the convergence speed.
As an application, we extend the convergence analysis
to a popular SCF variant -- the level-shifted SCF.
The effectiveness of the convergence rate estimates is
demonstrated numerically for NEPv arising from 
solving the Kohn-Sham equation in electronic structure calculation and
the Gross-Pitaevskii equation in the modeling of Bose-Einstein condensation.
\end{abstract}




\section{Introduction}
We consider the following
nonlinear eigenvalue problem with eigenvector-dependency (NEPv):
find an orthonormal matrix $V\in \bbC^{n\times k}$, i.e.,
$V^{\HH}V = I_k$, and
a square matrix $\Lambda\in\bbC^{k\times k}$ satisfying
\begin{equation}\label{eq:nepv}
    H(V)V=V\Lambda,
\end{equation}
where $H\colon \bbC^{n\times k}\to \bbC^{n\times n}$ is a continuous
Hermitian matrix-valued function of $V$.
Necessarily, $\Lambda = V^{\HH} H(V) V$ and the eigenvalues 
of $\Lambda$ are $k$ eigenvalues of $H(V)$, often
either the $k$ smallest or largest ones. Our later analysis 
will focus on $\Lambda$ 
associated with the $k$ smallest eigenvalues of $H(V)$, 
but it works equally well for the case when
$\Lambda$ is associated with the $k$ largest ones. 
We assume throughout this paper 
that $H(V)$ is right-unitarily invariant in $V$, i.e.,
\begin{equation}\label{eq:univar}
H(VQ)=H(V) \quad \mbox{for any unitary $Q\in\bbU^{k\times k}$},
\end{equation}
where $\bbU^{k\times k}$ is the set of all $k\times k$ unitary matrix.
This property~\eqref{eq:univar} essentially says that 
NEPv~\eqref{eq:nepv} is eigenspace-dependent,
to be more precise.
However, we will adopt the notion of nonlinear eigenvalue problem 
with eigenvector-dependency commonly used in literature.
Furthermore,  the assumption~\eqref{eq:univar} implies that
if $(V,\Lambda)$ is a solution of  NEPv~\eqref{eq:nepv},
then so is  $(VQ,Q^{\HH}\Lambda Q)$ for any unitary $Q$.
We therefore view $V$ and $\widetilde V$
as an identical solution, if the two share 
a common range $\cR(V)=\cR(\widetilde V)$.


NEPv in the form of~\eqref{eq:nepv} arises frequently in
a number of areas of computational science and engineering.  
They are the discrete representations of
the Kohn-Sham equation of the density functional theory in electronic structure
calculations~\cite{Martin:2004,Szabo:2012},
and the Gross-Pitaevskii equation in modeling the ground state wave function
in a Bose-Einstein condensate~\cite{Bao:2004,Jarlebring:2014}.
In particular, $H(V) = \Phi(P)$, where $\phi$
is a Hermitian matrix-valued function of $P=VV^{\HH}$, known as 
the density matrix in the density functional 
theory~\cite{Martin:2004,Szabo:2012}.
NEPv have also long played important roles in the classical 
methods for data analysis, such as multidimensional scaling \cite{Meyer:1997}.
It has become increasingly popular recently in the fields of 
machine learning and network science, such as 
the trace ratio maximizations for dimensional 
reduction~\cite{Ngo:2012,Zhang:2014}, 
balanced graph cut~\cite{Jost:2014},
robust Rayleigh quotient maximization for 
handling data uncertainty~\cite{Bai:2018}, 
core-periphy detection in networks~\cite{tuhi:2019}, 
and orthogonal canonical correlation analysis \cite{zhwb:2020}. 
The unitary invariance~\eqref{eq:univar} holds in all those 
practical NEPv except few. 


The Self-Consistent Fields (SCF) iteration
is the most general and widely-used method to 
solve NEPv~\eqref{eq:nepv}.
SCF, first introduced in molecular quantum mechanics back 
to 1950s~\cite{Roothaan:1951}, serves as an entrance to 
all other approaches. 
Starting with an orthonormal matrix $V_0\in\bbU^{n\times k}$,
SCF computes iteratively $V_{i+1}$ and $\Lambda_{i+1}$ satisfying
\begin{equation}\label{eq:pscf}
H(V_i) V_{i+1} = V_{i+1}\Lambda_{i+1},
\quad\text{for}\quad i = 0,1,2,\dots,
\end{equation}
where $V_{i+1}\in\bbC^{n\times k}$ is orthonormal and
$\Lambda_{i+1}$ is a diagonal matrix consisting of the $k$ smallest
eigenvalues of $H(V_i)$.
Since unit eigenvectors associated with simple eigenvalues can 
differ by scalar factors of unimodular complex numbers
and those associated with multiple eigenvalues have even more freedom,
the iteration matrix $V_{i+1}$ cannot be uniquely defined.
But thanks to the property~\eqref{eq:univar}, the computed subspaces
$\cR(V_1),\cR(V_2),\dots $ are always the same,
provided the $k$th and $(k+1)$st eigenvalues of $H(V_i)$ are 
distinct at the $i$th
iteration. Because of this, SCF can be interpreted as an iteration
of subspaces of dimension $k$, i.e., elements in the Grassmann
manifold $\mathbf{Gr}(k,\bbC^n)$ of all $k$ dimensional subspaces of $\bbC^n$.


The procedure in~\eqref{eq:pscf} is an SCF in its simplest form,
also known as the plain SCF iteration.
In practice, such a procedure is prone to slow convergence 
and sometimes may not converge~\cite{Koutecky:1971}.
Therefore it has been a fundamental problem of 
intensive research for decades to understand when and how the plain SCF
would converge, so as to develop remedies to stabilize and 
accelerate the SCF iteration.


For the applications of solving the Kohn-Sham equation in physics and quantum
chemistry, the solution of the associated NEPv corresponds to the minimizer 
of an energy function.
In such context, optimization techniques can be employed  to establish
convergence results of SCF. 
A number of convergence conditions have been 
investigated~\cite{Cances:2000,Liu:2014,Liu:2015,Yang:2009}.
For solving general NEPv, one may view the plain SCF~\eqref{eq:pscf} 
as a simple fixed-point iteration.
Sufficient conditions for the fixed-point map being a contraction,
in terms of the sines of the canonical angles between subspaces,
has been studied in~\cite{Cai:2018}, where the authors revealed 
a convergence rate for SCF based on the Davis-Kahan Sin$\Theta$ 
theorem for eigenspace perturbation~\cite{Davis:1970}.
Another approach for the fixed-point analysis is to look at the 
spectral radius of
the Jacobian supermatrix of the fixed-point map, be it differentiable.
When $H(V)$ is a smooth function explicitly in the density matrix $P=VV^{\HH}$,
a closed-form expression of the Jacobian has been obtained in a
recent work~\cite{Upadhyaya:2018}. Similar analysis also appeared
in an earlier work \cite{Stanton:1981} by focusing on 
the Hartree--Fock equation.


What is often different in the existing convergence analysis
is the way of measuring the approximation error.
Since SCF is a subspace iteration,
how to assess the distance between two subspaces $\cR(V)$ and $\cR(V_*)$
is the key to the convergence analysis.
Various distance measures have been applied in the literature,
leading to different approaches of analysis and different types of convergence
results. In particular, 
the difference in density matrices in 2-norm is
used as a measure of distance in~\cite{Yang:2009}; 
A chordal 2-norm is used in~\cite{Liu:2014}; 
More recent work \cite{Cai:2018} turned to the sines 
of the canonical angles between subspaces;
The work~\cite{Cances:2000} as well as \cite{Upadhyaya:2018} though 
not explicitly specified, used the difference of density matrices 
in the Frobenius norm. We believe that those distance measures 
may not necessarily be the best to capture 
the intrinsic feature of the SCF iteration. 


The results presented in this paper is a refinement and extension 
of the previous ones in~\cite{Cai:2018,Liu:2014,Upadhyaya:2018,Yang:2009}.
We aim to provide a comprehensive and unified local convergence analysis 
of SCF.
Rather than resorting to a specific distance measure, our development is based
on the tangent-angle matrix, associated with the tangents of canonical 
angles of subspaces. Such matrices can precisely capture the error
recurrence of SCF when close to convergence,
and they can act as intermediate measurements,
by which various distance measures can be evaluated as needed.
Despite less popular than sines, the tangents of canonical angles
have also been used to assess the distance between subspaces,
and can lead to tighter bounds when 
applicable, see~\cite{Davis:1970,Zhu:2013}, and references therein.

The use of tangent-angle matrix allows us to take a closer 
examination at the local error recursion of SCF, leading 
to the following new contributions presented in this paper: 
\begin{enumerate}[(a)]
\item
A precise characterization for the local contraction
factor of SCF for both continuous and differentiable $H(V)$. 
This improves over the previous work~\cite{Cai:2018,Liu:2014,Yang:2009}, 
where only upper bounds of such a quantity were obtained.

\item
A closed-form formula for the local asymptotic average contraction factor 
of SCF in terms of the spectral radius of an underlying linear operator 
when $H(V)$ is differentiable.
The formula is optimal for providing 
a sufficient and almost necessary local convergence 
condition of SCF.  
It extends the previous work in~\cite{Stanton:1981,Upadhyaya:2018}
to general $H(V)$ functions, and has a compact
expression that is convenient to work with in both theory and computation.

\item
A new justifications for a commonly used 
level-shifting scheme for the stabilization and acceleration 
of SCF~\cite{Cances:2000}. A closed-form lower 
bound on the shifting parameter to guarantee local convergence
is obtained.
\end{enumerate}

The rest of the paper is organized as follows.
\Cref{sec:prelim} presents some preliminaries to set up basic definitions 
and assumptions.
\Cref{sec:tangent} introduces the tangent-angle matrix and
establishes the  recurrence relation of such matrices in consecutive SCF
iteration.
\Cref{sec:conv} is devoted to the local convergence theory 
of the plain SCF iteration.
\Cref{sec:ls} deals with the level-shifted SCF and its convergence.
Numerical illustrations are in~\Cref{sec:examples},
followed by conclusions in~\Cref{sec:conclusion}.


We follow the notation convention in matrix analysis:
$\bbR^{m\times n}$ and $\bbC^{m\times n}$ are the sets of 
$m\times n$ real and complex matrices, respectively,
and $\bbR^n=\bbR^{n\times 1}$ and $\bbC^n=\bbC^{n\times 1}$.
$\bbU^{m\times n}\subset \bbC^{m\times n}$ denotes the set of 
$m\times n$ complex orthonormal matrices.
$A$, $A^{\T}$ and $A^{\HH}$ are the transpose and  conjugate
transpose of a matrix or a vector $A$, respectively, and $\overline{A}$ takes entrywise conjugate.
$H_1\succeq H_2$ means that $H_1$ and $H_2$ are Hermitian 
matrices and $H_1-H_2$ is positive semi-definite.
For a matrix $H\in\bbC^{n\times n}$ known to have real eigenvalues only, 
$\lambda_i(H)$ is the $i$th eigenvalue of $H$ in the
ascending order, i.e., $\lambda_1(H)\leq \lambda_2(H)\leq\cdots\leq \lambda_n(H)$, and
$\lambda_{\min}(H)=\lambda_1(H)$ and
$\lambda_{\max}(H)=\lambda_n(H)$.
$\Diag(x)$ is a diagonal matrix formed by the vector $x$,
$\diag (X)$ is a vector consisting of the diagonal elements of a matrix $X$;
$\cR(X)$ is the range of $X$;
$\sigma(X)$ is the collection of all singular values of $X$.
$\Re(\cdot)$ and $\Im(\cdot)$ extract the real and imaginary parts of
a complex number and, when applied to a matrix/vector, 
they are understood in the elementwise sense.
Standard big-O and little-o notations in mathematical analysis are used:
for functions $f(x), g(x)\to 0$ as $x\to 0$,
write $f(x) = \calo(g(x))$ if $|f(x)|\leq c|g(x)|$ 
for some constant $c$ as $x\to 0$,
and write $f(x) = \lo(g(x))$ if $|f(x)|/|g(x)| \to 0$ as $x\to 0$.
Other notations will be explained at their first appearance.

\section{Preliminaries}\label{sec:prelim}

Throughout this paper, we denote by $V_*\in\bbU^{n\times k}$
a solution of  NEPv~\eqref{eq:nepv}. 
The eigen-decomposition of $H(V_*)$ is given by
\begin{equation}\label{eq:eigdec}
H(V_*)\,[V_*,V_{*\bot}] = [V_*,V_{*\bot}]\,\twobytwo{\Lambda_*}{}{}{\Lambda_{*\bot}},
\end{equation}
where $ [V_*,V_{*\bot}]\in \bbU^{n\times n}$ is unitary,
\[
\text{$\Lambda_*=\diag(\lambda_1,\dots,\lambda_k)$
and
$\Lambda_{*\bot}=\diag(\lambda_{k+1}, \dots, \lambda_n)$}
\]
are diagonal matrices containing the  eigenvalues
of $H(V_*)$ in the ascending order, i.e., $\lambda_i=\lambda_i(H(V_*))$.
We make the following assumption for the solution $V_*$ of 
NEPv~\eqref{eq:nepv} under consideration.

\begin{assumption}\label{ass:gap}
There is a positive eigenvalue gap:
\begin{equation}\label{eq:gap}
\delta_*:=\lambda_{k+1}(H(V_*))-\lambda_k(H(V_*)) > 0.
\end{equation}
\end{assumption}

Such an assumption,
which is commonly applied in the convergence analysis of
SCF guarantees the uniqueness of the eigenspace corresponding
to the $k$ smallest eigenvalues of 
$H(V_*)$~\cite{Cai:2018,Cances:2000,Liu:2014,Upadhyaya:2018,Yang:2009}.

\paragraph{Sylvester equation}
The following Sylvester equation in $X\in\bbC^{n\times k}$ will be needed in
our analysis
\begin{equation} \label{eq:sylvester}
\Lambda_{*\bot} X - X\Lambda_*=V_{*\bot}^{\HH}[H(V_*)-H(V)]V_*.
\end{equation}
Under~\Cref{ass:gap}, this equation has  a unique solution $X\equiv \SL(V)$
for each $V\in\bbU^{n\times k}$, given by
\begin{equation}\label{eq:xv}
\SL(V) = D(V_*) \odot \left(V_{*\bot}^{\HH}[H(V_*) - H(V)]V_*\right),
\end{equation}
where
\begin{equation}\label{eq:dvs}
D(V_*)\in\bbR^{(n-k)\times k}
\quad \text{with $D(V_*)_{ij} = (\lambda_{k+i}(H(V_*))-    \lambda_{j}(H(V_*)))^{-1}$},
\end{equation}
and $\odot$ denotes the Hadamard product, i.e. elementwise multiplication.

\paragraph{Unitarily invariant norm}
We denote by $\|\cdot\|_{\UI}$ a unitarily invariant norm,
which, besides being a matrix norm, also satisfies
the following two additional conditions:
\begin{enumerate}[(1)] 
    \item
        $\|XAY\|_{\UI} = \|A\|_{\UI}$ for any unitary matrices $X$ and $Y$;
    \item
        $\|A\|_{\UI} = \|A\|_2$ whenever $A$ is rank-1, where $\|\cdot\|_2$ is the
        spectral norm.
\end{enumerate}
It is well-known that $\|A\|_{\UI}$ is dependent
only on the singular values of $A$.
In this paper, we assume any $\|\cdot \|_{\UI}$ we use
is applicable to matrices of all sizes in
a compatible way, i.e., $\|A\|_{\UI} =\|B\|_{\UI}$ for $A$, $B$ sharing
a same set of non-zero singular values 
(see, e.g., \cite[Thm 3.6, pp 78]{Stewart:1990}).
The spectral norm $\|\cdot\|_2$ and Frobenius norm $\|\cdot \|_{\F}$
are two particular examples of such unitarily invariant norms.

\paragraph{Canonical angles between subspaces}
Let $X, Y\in\bbU^{n\times k}$.
The $k$ canonical angles between the range spaces of $\mathcal X=\cR(X)$ and
$\mathcal Y = \cR(Y)$ are defined as
\begin{equation}\label{eq:indv-angles-XY}
0\le\theta_j(\mathcal X,\mathcal Y):=\arccos\sigma_j\le\frac {\pi}2\quad\mbox{for $1\le j\le k$},
\end{equation}
where $\sigma_1\ge\cdots\ge\sigma_k$ are singular values of
the matrix $Y^{\HH}X$ (see, e.g.,~\cite[Sec 4.2.1]{Stewart:1990}).
Put $k$ canonical angles all together to define
\begin{equation}\label{eq:mat-angles-XY}
    \Theta(\mathcal X,\mathcal Y)=\diag(\theta_1(\mathcal X,\mathcal
    Y),\ldots,\theta_k(\mathcal X,\mathcal Y)).
\end{equation}
Since the canonical angles so defined are independent of the basis matrices $X$ and
$Y$,  for convenience, we use the notation
$\Theta({X},{Y})$ interchangeably with $\Theta(\mathcal X,\mathcal Y)$.

Canonical angles provide a natural distance measure for subspaces.
For any unitarily invariant norm $\|\cdot\|_{\UI}$,
it holds that both $\|\Theta({X},{Y})\|_{\UI}$ and $\|\sin\Theta({ X},{ Y})\|_{\UI}$ are
unitarily invariant metrics on the Grassmann manifold $\mathbf {Gr}(k,\bbC^n)$
(see e.g.,~\cite[Thm. 4.10, pp 93]{Stewart:1990} and~\cite{Qiu:2005}).
In our analysis, the tangents of canonical angles will play an important role.
By trigonometric function analysis,
tangents provide good approximation to the canonical angles as 
$\Theta(X,Y)\to 0$:
\begin{equation}\label{eq:ts}
    \tan\Theta({X},{Y}) = \Theta({X},{Y}) + \calo(\|\Theta({X},{Y})\|_{\UI}^3).
\end{equation}

\paragraph{$\bbR$-linear mapping}
A mapping $\OP\colon \bbC^{n\times k}\to \bbC^{p\times q}$ is called
$\bbR$-linear, if it satisfies
\begin{equation}\label{eq:rlinear}
    \OP(X+Y) = \OP(X) + \OP(Y)
    \quad\text{and}\quad
    \OP(\alpha\,  X) = \alpha\,  \OP(X)
\end{equation}
for all $X, Y\in\mathbb C^{n\times k}$ and $\alpha\in\mathbb R$.
When we talk about an $\mathbb R$-linear mapping,
the complex matrix space $\bbC^{m\times n}$ is viewed as a vector space over
the field $\bbR$ of real numbers, denoted by $\bbC^{m\times n}(\bbR)$.
By elementary linear algebra, $\bbC^{m\times n}(\bbR)$ is a $(2mn)$-dimensional
inner product space, equipped with 
the inner product $\langle X,Y\rangle:= \Re \tr(X^{\HH}Y)$
and the induced norm $\|X\|_{\F} = \left(\Re \tr(X^{\HH}X)\right)^{1/2}$. 
We can see that
$\OP\colon \bbC^{n\times k}(\bbR)\to \bbC^{p\times q}(\bbR)$
is a linear mapping (over $\bbR$).
For convenience, we use $\mathbb C^{n\times k}$ and $\mathbb C^{n\times
k}(\bbR)$ interchangeably in future discussions when referring
to an $\bbR$-linear mapping.

The \emph{spectral radius} of
an $\bbR$-linear operator $\OP: \bbC^{n\times k} \to \bbC^{n\times k}$
is defined as the largest eigenvalue in magnitude of a 
matrix representation
$\mathbf L\in\bbR^{(2nk)\times (2nk)}$ of $\OP$: 
\begin{equation}\label{eq:sprad}
\rho(\OP) := \max \left\{\ |\lambda|\colon \mathbf L \, \mathbf x =
    \lambda \,  \mathbf x,\, \mathbf x\in\bbC^{2nk}\ \right\}.
\end{equation}
Notice that $\mathbf x$ is allowed to be a complex vector,
because a real matrix can have complex eigenvalues.
Here we do not make any assumption on the basis used to obtain $\mathbf L$,
the choice of the basis does not affect the spectrum of $\mathbf L$, 
therefore $\rho(\OP)$.

\paragraph{Derivative operator}
Let $V=V_r+\imath V_i\in \bbC^{n\times k}$ with
$V_r,V_i\in\bbR^{n\times k}$ being the real and
imaginary parts of $V$, respectively.
A Hermitian matrix-valued function $H(V)$ is called differentiable, if each
element $h_{ij}(V)$ is a smooth function in the real and imaginary parts
$(V_r,V_i)$ of $V$.
Such differentiability is different from the one in the
holomorphic sense, which generally cannot hold for $H(V)$ with real
diagonal elements.
For  $H(V)$ differentiable at $V_*$, 
we can define a derivative operator
\begin{equation}\label{eq:dh}
    \text{ $\mbox{\bf D}H(V_*)[\cdot ]\colon \bbC^{n\times k} \to
    \bbC^{n\times n}$ \quad with\quad }
    \mbox{\bf D}H(V_*)[X] = \left[\frac{d}{dt} H(V_*+t
    X)\right]_{t=0},
\end{equation}
where $t\in\bbR$.
$\mbox{\bf D}H(V_*)[X]$ represents the
derivative of $H(V)$ at $V_*$,
in the direction of $X\in\mathbb C^{n\times k}$.
A direct verification shows that
$\mbox{\bf D}H(V_*)[\cdot ]$
is an $\bbR$-linear mapping satisfying \eqref{eq:rlinear}.

By Taylor's expansion of $H(V_*+tX)$ at $t=0$,
as $V$ close to $V_*$ (in the Euclidean sense),
it holds
\begin{equation}\label{eq:taylor}
H(V) = H(V_*) + \mbox{\bf D}H(V_*)[V - V_*] + o (\|V-V_*\|_2).
\end{equation}
Therefore, $\mbox{\bf D}H(V_*)[\cdot ]$
is the Fr\'echet derivative of
$H: \bbC^{n\times k}(\bbR)\to \bbC^{n\times n}(\bbR)$.
Note that the expansion~\eqref{eq:taylor} does not take into account
the unitary invariance \eqref{eq:univar} of $H(V)$, and that is
why the remainder term is in the Euclidean difference $V-V_*$.

\section{Tangent-angle matrix}\label{sec:tangent}

Let $V\in \bbU^{n\times k}$ be an approximation to the solution $V_*$ of
NEPv~\eqref{eq:nepv}. Each $V$ represents an orthonormal basis matrix 
of a subspace.
As far as a solution of NEPv~\eqref{eq:nepv} is concerned, it is the subspaces
that matter.
To assess the distance of $V$ to the solution  $V_*$ in
terms of the subspaces their columns span,
we define the {\em tangent-angle matrix} from $V$ to $V_*$ as
\begin{equation} \label{eq:tv}
T(V) := (V_{*\bot}^{\HH}V)(V_*^{\HH} V)^{-1}\in  \bbC^{(n-k)\times k},
\end{equation}
provided $V_*^{\HH} V$ is invertible. 
By definition, $T(V)$ can be viewed as a function of 
$\bbU^{n\times k} \to \bbC^{(n-k)\times k}$.
The name of `tangent-angle matrix' comes from
the fact that
\begin{equation}\label{eq:tmat}
    \|\tan\Theta(V,V_*)\|_{\UI} = \|(V_{*\bot}^{\HH}V)(V_*^{\HH} V)^{-1}\|_{\UI}
    = \|T(V)\|_{\UI},
\end{equation}
for all unitarily invariant norms.
Recall that
the unitarily invariant norm
$\|A\|_{\UI}$ is defined by the singular values of $A$,
equation~\eqref{eq:tmat} is a direct consequence of the 
identity of singular values
$
\sigma(\tan \Theta(V,V_*)) = \sigma\left((V_\perp^{\HH} V_*)(V^{\HH}
V_*)^{-1}\right)$,
which follows from the definition of canonical angles in~\eqref{eq:mat-angles-XY}
(see, e.g., \cite[Thm. 2.2, 2.4, Chap 4]{Stewart:2001} and~\cite{Zhu:2013}).
The tangents of canonical angles have long been used in numerical matrix 
analysis,
and we refer to~\cite{Zhu:2013} and references therein.

By definition~\eqref{eq:indv-angles-XY},
the singular values of $V_*^{\HH}V$ consist of
those of the matrix
$\cos\Theta(V,V_*) = I + \calo(\|\Theta(V,V_*)\|_{\UI}^2)$.
Therefore, it can be seen from~\eqref{eq:tmat} that $T(V)$ 
is well defined for sufficiently small canonical angles $\Theta(V,V_*)$.  
Meanwhile, $\Theta(V,V_*)\to 0$ iff $T(V)\to 0$.
By the unitary invariance \eqref{eq:univar} and 
the continuity of $H(V)$,
we have $H(V)\to H(V_*)$ as the tangent-angle matrix $T(V)\to 0$.
This is more precisely described in the following lemma.

\begin{lemma}\label{lem:hv}
Let $V\in\bbU^{n\times k}$. Then as $T(V)\to 0$, it holds that
\begin{equation}\label{eq:uninvar}
H(V) = H\big( V_* + V_{*\bot}T(V) + \calo(\|T(V)\|_{\UI}^2)\big).
\end{equation}
If $H(V)$ is also differentiable, then
\begin{equation}\label{eq:diff}
H(V)  =  H(V_*) + {\bf D}H(V_*)[V_{*\bot}T(V)] + \lo(\|T(V)\|_{\UI}).
\end{equation}
\end{lemma}
\begin{proof}
The singular values of $V_*^{\HH}V$ consist of
$\cos\Theta(V,V_*) = I + \calo(\|\Theta(V,V_*)\|_{\UI}^2) $.
So we have
$V_*^{\HH}V = W + \calo(\|\Theta(V,V_*)\|_{\UI}^2)$ for some
unitary $W \in\bbU^{k\times k}$. 
It follows that
\begin{equation} \label{eq:vw}
VW^{-1} = V(V_*^{\HH}V)^{-1} + \calo(\|\Theta(V,V_*)\|_{\UI}^2)
    = V_* + V_{*\bot}\,  T(V) + \calo(\|T(V)\|_{\UI}^2),
\end{equation} 
where we used $V= V_*(V_*^{\HH}V)+V_{*\bot}(V_{*\bot}^{\HH}V)$
and $T(V) = \calo(\|\Theta(V,V_*)\|_{\UI})$ in the last equation.
The unitary invariance property $H(V) = H(VW^{-1})$ leads to~\eqref{eq:uninvar}.
Combining \eqref{eq:vw} with~\eqref{eq:taylor}, 
we obtain~\eqref{eq:diff}.
\end{proof}

The following lemma, which is the key to establishing our local convergence
results, describes the relation between the tangent-angle matrices of
two consecutive SCF iterations.

\begin{lemma}\label{lem:onestep}
Suppose Assumption~\ref{ass:gap} holds.
Let $\widetilde V$ be an orthonormal basis matrix associated with 
the $k$ smallest eigenvalues of $H(V)$,
and let $\SL(V)$ be the unique solution of the Sylvester equation
defined in~\eqref{eq:xv}.
Then 
\begin{enumerate}[{\rm (a)}]
\item \label{item:tx} 
$\SL(V)\to 0$ as $T(V)\to 0$;

\item \label{item:tv}
the tangent-angle matrix $T(\widetilde V)$ of $\widetilde V$ satisfies
            \begin{equation}\label{eq:onestep}
                T(\widetilde V) = \SL(V) + \lo( \|\SL(V)\|_{\UI});
            \end{equation}

\item \label{item:ls}
            if $H(V)$ is differentiable at $V_*$, then
            \begin{equation}\label{eq:onestepdiff}
                T(\widetilde V) = \OP\left(T(V)\right) + \lo(\|T(V)\|_{\UI}),
            \end{equation}
            where $\OP:\mathbb C^{(n-k)\times k}\to \mathbb C^{(n-k)\times k}$ defined by
\begin{equation}\label{eq:ls}
\OP(Z) =  D(V_*)\odot (V_{*\bot}^{\HH} \, \mbox{\bf D}H(V_*)[V_{*\bot} Z]\,  V_*)
\end{equation}
is an $\mathbb R$-linear operator, called the {\em\LOP} of the plain SCF.
    \end{enumerate}
\end{lemma}
\begin{proof}
For item~\ref{item:tx},  
by~\eqref{eq:uninvar} and the continuity of $H$,
it holds $H(V) \to H(V_*)$ as $T(V)\to 0$.
Hence, $\SL(V)\to 0$ by the definition of $\SL(V)$.

For item~\ref{item:tv},
we begin with the eigen-decomposition of $H(V)$:
\[
    H(V) \begin{bmatrix} \widetilde V & \widetilde V_{\bot} \end{bmatrix}
    =
    \begin{bmatrix} \widetilde V& \widetilde V_{\bot} \end{bmatrix}
    \begin{bmatrix} \widetilde \Lambda & \\ &\widetilde \Lambda_{\bot}
    \end{bmatrix},
\]
where $ \begin{bmatrix} \widetilde V & \widetilde V_{\bot}
\end{bmatrix}\in\bbU^{n\times n}$ is unitary,
$\widetilde \Lambda=\diag(\widetilde \lambda_1,\dots, \widetilde \lambda_k)$ and
$\widetilde \Lambda_{\bot}=\diag(\widetilde \lambda_{k+1},\dots, \widetilde \lambda_n)$ with
$\widetilde \lambda_i=\lambda_i(H(V))$.
Due to \Cref{ass:gap}, as $H(V) \to H(V_*)$,
we can apply the standard perturbation analysis of eigenspaces~\cite[Sec.~V.2]{Stewart:1990} to obtain
\begin{equation}\label{eq:paras:VtV}
    \begin{bmatrix} \widetilde V & \widetilde V_{\bot} \end{bmatrix}
        =
    \begin{bmatrix} V_* & V_{*\bot} \end{bmatrix}
    \begin{bmatrix} I_k & - Z^{\HH} \\ Z & I_{n-k}\end{bmatrix}
    \begin{bmatrix} (I_k+Z^{\HH}Z)^{-1/2} & \\ & (I_{n-k}+ZZ^{\HH})^{-1/2}
    \end{bmatrix}
    \begin{bmatrix}
        Q & \\ & P
    \end{bmatrix},
\end{equation}
where $Z\in\bbR^{(n-k)\times k}$,
$Q\in\bbU^{k\times k}$, and $P\in\bbU^{(n-k)\times (n-k)}$
are parameter matrices, and
\begin{equation}\label{eq:zt0}
    \text{ $Z\to 0$ \quad as\quad $H(V)\to H(V_*)$}.
\end{equation}
The parameterization from \eqref{eq:paras:VtV} can be equivalently put as
\begin{align*}
    \widetilde V  &= (V_* + V_{*\bot} Z)\,  (I_k+Z^{\HH}Z)^{-1/2}Q,\\
    \widetilde V_{\bot}  &= (-V_*Z^{\HH} + V_{*\bot})\,      (I_{n-k}+ZZ^{\HH})^{-1/2} P.
 \end{align*}
By the first equation,
$Z$ is identical to the tangent-angle matrix from $\widetilde V$ to $V_*$:
\begin{equation}\label{eq:z}
    T(\widetilde V) =
    ( V_{*\bot}^{\HH} \widetilde V) (V_{*}^{\HH} \widetilde V)^{-1} =Z,
\end{equation}
where we have used $V_{*}^{\HH} \widetilde V = (I_k+Z^{\HH}Z)^{-1/2}Q$ and $V_{*\bot}^{\HH}
\widetilde V = Z(I_k+Z^{\HH}Z)^{-1/2}Q$.

Next, we establish an equation to characterize $Z$.
From $\widetilde V_{\bot}^{\HH} H(V) \widetilde V= \widetilde V_{\bot}^{\HH}
\widetilde V\widetilde \Lambda = 0$,
we get
\begin{align*}
    0&=\begin{bmatrix} -Z & I_{n-k} \end{bmatrix}[V_*, V_{*\bot}]^{\HH}
    H(V) [V_*, V_{*\bot}]\begin{bmatrix}I_k \\ Z\end{bmatrix} \\
    &=\begin{bmatrix} -Z & I_{n-k} \end{bmatrix}[V_*, V_{*\bot}]^{\HH}
    [H(V_*) + (H(V) - H(V_*))] [V_*, V_{*\bot}]\begin{bmatrix}I_k \\
    Z\end{bmatrix}\\
    &=\Lambda_{*\bot}Z-Z\Lambda_*+(-ZV_*^{\HH}+V_{*\bot}^{\HH})[H(V) - H(V_*)](V_*+V_{*\bot}Z).
\end{align*}
Therefore, $Z$ satisfies the Sylvester equation (view the right hand side as fixed)
\[ \Lambda_{*\bot}Z-Z\Lambda_* =
    (ZV_*^{\HH}-V_{*\bot}^{\HH})[H(V) - H(V_*)](V_*+V_{*\bot}Z).
\]
By \Cref{ass:gap}, we can solve the Sylvester equation to obtain
\begin{align}\label{eq:zxphi}
    Z= \SL(V) + \Phi(Z),
\end{align}
where
\[
    \Phi(Z)=D(V_*)\odot\left(ZV_*^{\HH}[H(V) - H(V_*)](V_*+V_{*\bot}Z)  -
    V_{*\bot}^{\HH}[H(V) - H(V_*)]V_{*\bot}Z\right),
\]
and $D(V_*)$ is defined as in~\eqref{eq:dvs}.
A quick calculation  shows that
\begin{equation}\label{eq:phiz}
    \|\Phi(Z)\|_{\F}\le \delta_*^{-1} \,   \|H(V) -    H(V_*)\|_{\F}(2\|Z\|_2+\|Z\|_2^2)
    =\lo( \|Z\|_{\UI}),
\end{equation}
where the last equation is due to $H(V)\to H(V_*)$ and $Z\to 0$, as $T(V)\to 0$,
and the equivalency of matrix norms.
Recall $T(\widetilde V) = Z$.  \Cref{eq:zxphi,eq:phiz} lead directly to~\eqref{eq:onestep}.


For item~\ref{item:ls}, we derive from the definition of $\SL(V)$ and
the expansion~\eqref{eq:diff} that
\begin{align*}
\SL(V) 
     & = D(V_*) \odot \left(V_{*\bot}^{\HH}\, {\bf D}H(V_*)[V_{*\bot}T(V)]\,  V_*\right)  +
     \lo(\|T(V)\|_{\UI}).
\end{align*}
Plugging it into~\eqref{eq:onestep}, and
exploiting $\|\OP(T(V))\|_{\UI} = \calo(\|T(V)\|_{\UI})$
since $\OP$ is an $\bbR$-linear operator of finite dimension
(which is bounded),  we complete the proof.
\end{proof}

We should mention that the tangent-angle matrix in the form of~\eqref{eq:tv}
appeared in the so-called
McWeeny transformation~\cite{Mcweeny:1960,Stanton:1981b,Stanton:1981}
in the density matrix theory for electronic structure calculations, where
the matrix was treated as an independent parameter that is not connected
with canonical angles of subspaces. This lack of geometric interpretation
makes it difficult to proceed a comprehensive convergence analysis
as developed in the following sections, and extend to 
the treatment of a continuous $H(V)$.

\section{Convergence analysis}\label{sec:conv}
Because of the invariance property \eqref{eq:univar},
the plain SCF iteration~\eqref{eq:pscf} should be inherently understood as a
subspace iterative scheme and the convergence of the basis matrices
$\{V_i\}_{i=0}^\infty$ to a solution $V_*$ should be measured by a metric on the
Grassmann manifold $\mathbf{Gr}(k,\bbC^n)$.  Let $d(\cdot,\cdot)$ be a metric
on $\mathbf{Gr}(k,\bbC^n)$. Without causing any ambiguity, in what follows we will not
distinguish an element $\cR(V)\in\mathbf{Gr}(k,\bbC^n)$ from  its representation $V\in\bbU^{n\times k}$.
The following notions are straightforward extensions of the existing ones:
\begin{enumerate}[(i)]
\item
SCF~\eqref{eq:pscf} is {\em locally convergent} to $V_*$, if
$d(V_i, V_*)\to 0$ as $i\to \infty$ for any initial $V_0$ that is sufficiently close to $V_*$ in
the metric, i.e., $d(V_0,V_*)$ is sufficiently small.

\item SCF~\eqref{eq:pscf} is {\em locally divergent} from
$V_*$, if for all $\varepsilon> 0$
there exists $V_0$ with $d(V_0,V_*)\leq \varepsilon$ such that
$d(V_i,V_*)$ doesn't converge to $0$, i.e., 
either $d(V_i,V_*)$ doesn't converge at all or converges to something not $0$.
\end{enumerate}

\subsection{Contraction factors}\label{sec:local}

There are two fundamental quantities
that provide convergence measures of SCF on $\mathbf{Gr}(k,\bbC^n)$:
{\em local contraction factor\/} and
{\em local asymptotic average contraction factor}.
The former, which is a quantity to assess local convergence,
accounts for the worst case error reduction of SCF per iterative
step.  The latter captures the asymptotic average convergence rate of SCF,
and provides a sufficient and almost necessary condition for
the local convergence.



Since SCF is a fixed-point iteration on the
Grassmann manifold $\mathbf{Gr}(k,\bbC^n)$,
the {\em local contraction factor of SCF} is defined as
\begin{equation}\label{eq:bestlip}
    \eta_{\sup}
    := \limsup_{V_0\in \bbU^{n\times k} \atop d(V_0,V_*)\to 0}
    \frac{ d(V_1,V_*)}{d(V_0,V_*)}.
\end{equation}
Such a constant can be viewed as the (best) local Lipschitz constant for
the fixed-point mapping of SCF.
We observe that the condition $\eta_{\sup}<1$, which implies SCF is locally error
reductive, is sufficient for local convergence.
In the convergent case, it follows from the definition~\eqref{eq:bestlip} that
\[
\limsup_{k\to \infty} \frac{d(V_{k+1},V_*)}{d(V_{k},V_*)} \leq \eta_{\sup},
\]
namely, the {\em (asymptotic) convergence rate} of SCF is bounded by
$\eta_{\sup}$.

To take into the account of oscillation
and to obtain tighter convergence bounds,
the one-step contraction factor \eqref{eq:bestlip}
can be generalized to multiple iterative steps.
Let $m$ be a given positive integer, and define
\begin{equation}\label{eq:mstepconv}
	\eta_{\sup,m} :=
    \limsup_{V_0\in \bbU^{n\times k} \atop  d(V_0,V_*)\to 0}
    \left(\frac{d(V_{m},V_*)}{d(V_0,V_*)}\right)^{1/m}.
\end{equation}
Then $\eta_{\sup,m}$ is an average contraction factor
per $m$ consecutive iterative steps of SCF~\eqref{eq:pscf}.
The limit of the average contraction factor,
as $m\to \infty$,
\begin{equation}\label{eq:assfac}
    \eta_{\sup,\infty} := \limsup_{m\to\infty} \ \eta_{\sup,m}
\end{equation}
defines a \emph{local asymptotic average contraction factor} of SCF.
By definition, the number $\eta_{\sup,\infty}$ measures
the average convergence rate of SCF.
The average convergence rate is a conventional tool to study matrix iterative
methods~\cite{Varga:1999} and typically leads to tight convergence rates
in practice.
It follows from item~\ref{i:lem:2sup:b} of the lemma below
that $\eta_{\sup,\infty}$ 
is the optimal local convergence rate and
thereby the optimal contraction factor of SCF.
We caution the reader that $\eta_{\sup}$, $\eta_{\sup,m}$ 
and $\eta_{\sup,\infty}$ 
depend on the metric $d(\cdot,\cdot)$ and the dependency is
suppressed for notational clarity. 

\begin{lemma}\label{lem:2sup}
Suppose~\Cref{ass:gap} and $\eta_{\sup}<\infty$.

\begin{enumerate}[{\rm (a)}]

\item\label{i:lem:2sup:a}
It holds that for any $m>1$
\begin{equation}\label{eq:eee}
\eta_{\sup,\infty}\leq  \eta_{\sup,m} \leq \eta_{\sup}.
\end{equation}

\item\label{i:lem:2sup:b}
If $\eta_{\sup,\infty} < 1$, then SCF is locally convergent to $V_*$, with its asymptotic average convergence rate bounded by $\eta_{\sup,\infty}$. If $\eta_{\sup,\infty} > 1$,  then SCF is locally divergent from $V_*$.


\end{enumerate}
\end{lemma}

\begin{proof}
For item~\ref{i:lem:2sup:a},
first from definition~\eqref{eq:bestlip} and $\eta_{\sup}<\infty$,
we conclude that ${d(V_p,V_*)}\to 0$ for $p=0,1,\dots, m-1$ as $d(V_0,V_*)\to 0$.
Therefore,
\begin{align*}
    \limsup_{V_0\in \bbU^{n\times k} \atop  d(V_0,V_*)\to 0}
    \left(\frac{d(V_{m},V_*)}{d(V_0,V_*)}\right)^{1/m} & =
    \limsup_{V_0\in \bbU^{n\times k} \atop  d(V_0,V_*)\to 0}
    \left(\prod_{p=0}^{m-1}\frac{d(V_{p+1},V_*)}{d(V_p,V_*)}\right)^{1/m} \\ 
    & \leq
    \left(\prod_{p=0}^{m-1}
    \limsup_{V_p\in \bbU^{n\times k} \atop  d(V_p,V_*)\to 0}
    \frac{d(V_{p+1},V_*)}{d(V_p,V_*)}
    \right)^{1/m},
\end{align*} 
    and $\eta_{\sup,m}\leq \eta_{\sup}$ follows.

    Now fix $m$. Any integer $m'>m$ can be expressed as $m' = sm + p$,
    for some $s\ge 0$ and $0\leq p\leq m-1$.
    Using the same arguments as from above, and noticing that
    \begin{align*}
    \left(\frac{d(V_{m'},V_*)}{d(V_0,V_*)}\right)^{1/m'}
    & =
    \left( \frac{d(V_{m'},V_*)}{d(V_p,V_*)} \frac{d(V_{p},V_*)}{d(V_0,V_*)} \right)^{1/m'} \\ 
    & =
    \left( \prod_{\ell = 0}^{s-1} \frac{d(V_{m(\ell+1)+p},V_*)}{d(V_{m\ell
    +p},V_*)} \cdot \frac{d(V_{p},V_*)}{d(V_0,V_*)} \right)^{1/m'},
    \end{align*}
    we obtain by taking $\limsup$ that
    \[
        \eta_{\sup,m'} \leq
        (\eta_{\sup,m})^{sm/m'}\cdot (\eta_{\sup,p})^{p/m'}
        =
        \eta_{\sup,m}\cdot
        \left(\eta_{\sup,p}\,/\,\eta_{\sup,m}\right)^{p/m'}.
    \]
    We can always assume $\eta_{\sup,m}\neq 0$, otherwise SCF converges
    in $m$ iterations and $\eta_{\sup,\infty}=0$.
    Letting $m'\to \infty$ and noticing that $\eta_{\sup,p}\leq \eta_{\sup}$ is bounded, we get
    $\eta_{\sup,\infty}=\limsup_{m'\to\infty}\eta_{\sup,m'}\leq \eta_{\sup,m}$.

For item~\ref{i:lem:2sup:b},
consider first $\eta_{\sup,\infty} < 1$. Pick  
a constant $c$ such that $\eta_{\sup,\infty} <c< 1$.
Because of how $\eta_{\sup,m}$ is defined in \eqref{eq:assfac}, we see that
$\eta_{\sup,m} \leq c$ for $m$ sufficiently large  and for 
all $V_0$ sufficiently close to $V_*$ in the metric $d(\cdot,\cdot)$.
Equivalently,
there exist $\delta_1 >0$ and $m_0>0$ such that
\begin{equation}\label{eq:2sup:pf-1}
d(V_{m},V_*) \leq c^{m} \, d(V_0,V_*)
\end{equation}
for all $V_0$ with $d(V_0,V_*) < \delta_1$ and for all $m\ge m_0$.
Recall that $\eta_{\sup}<\infty$ and pick a finite constant 
$c_2>\max\{1,\eta_{\sup}\}\ge 1$.
By \eqref{eq:bestlip}, 
there exists $\delta_2\in(0,\delta_1)$ such that
\begin{equation}\label{eq:2sup:pf-2}
d(V_1,V_*) \leq c_2 \, d(V_0,V_*)
\end{equation}
for all $V_0$ with $d(V_0,V_*) < \delta_2$.
Let $\delta_3= c_2^{-(m_0-1)}\times\delta_2<\delta_2<\delta_1$. 
For any $V_0$ with $d(V_0,V_*) < \delta_3$, we have by \eqref{eq:2sup:pf-2}
\begin{subequations}\label{eq:2sup:pf-3}
\begin{align}
d(V_1,V_*)&\leq c_2 \, d(V_0,V_*)<c_2\delta_3\le\delta_2, \label{eq:2sup:pf-3a}\\
d(V_2,V_*)&\le c_2 d(V_1,V_*)\le c_2^2d(V_0,V_*)<c_2^2\delta_3\le\delta_2, \label{eq:2sup:pf-3b}\\
   &\vdots \nonumber\\
d(V_{m_0-1},V_*)&\le c_2^{m_0-1}d(V_0,V_*)<c_2^{m_0-1}\delta_3\le\delta_2. \label{eq:2sup:pf-3c}
\end{align}
\end{subequations}
For any $m>m_0$, we can write $m=sm_0+p$ for some $0\le p\le m_0-1$.
We have by \eqref{eq:2sup:pf-1} and \eqref{eq:2sup:pf-3} that for $m>m_0$ and for any $V_0$ with $d(V_0,V_*) < \delta_3$
    \[
        d(V_{m},V_*)
        \leq c^{sm_0} \cdot d(V_p,V_*)
        =c^{m} \cdot \frac{ d(V_p,V_*)}{c^{p} }
        \leq c^{m} \cdot \frac{ \delta_2}{c^{m_0-1}}.
    \]
    Letting $m\to \infty$ yields $d(V_{m},V_*) \to 0$, as expected.

    On the other hand, if $\eta_{\sup,\infty} > 1$,
    then there exist $c > 1$ and a subsequence $\{m_i\}_{i=0}^\infty$ of positive integers such that
    $\eta_{\sup,m_i} \geq c$  as $i\to\infty$.
    Let $\delta>0$ be a constant satisfying $c-\delta >1$.
    It follows from the definition of $\eta_{\sup,m}$ that
    for all $\varepsilon> 0$ there exists $V_0$, with
    $d(V_0,V_*) \leq \varepsilon$, s.t.,
    $d(V_{m_i},V_*)/ d(V_0,V_*) \geq (c-\delta )^{m_i}$,
    which is arbitrarily large as $m_i\to\infty$.
    Hence the iteration is locally divergent.
\end{proof}

\subsection{Characterization of contraction factors}\label{sec:sprd}

The definitions of $\eta_{\sup}$ in~\eqref{eq:bestlip}
and $\eta_{\sup,\infty}$ in~\eqref{eq:assfac} are generic.
A meaningful characterization 
of $\eta_{\sup}$ and $\eta_{\sup,\infty}$ will have to involve
the specific choice of the metric $d(\cdot,\cdot)$ and
the detail of $H(V)$.
\Cref{thm:lnorm} below contains the main contributions of this paper.
It reveals for a class of metrics a direct characterization of $\eta_{\sup}$ by $H(V)$,
as compared to the previous works~\cite{Cai:2018,Liu:2014,Yang:2009}
on the upper bounds of $\eta_{\sup}$.
Furthermore, for differentiable $H(V)$,
it provides closed-form expressions for
$\eta_{\sup}$ and the optimal contraction factor
$\eta_{\sup,\infty}$.

\begin{theorem}\label{thm:lnorm}
Suppose~\Cref{ass:gap} and let $d(\cdot,\cdot):=\|\Theta(\cdot,\cdot)\|_{\UI}$.
\begin{enumerate}[{\rm (a)}]
\item \label{i:thm:lnorm:a1}
If $H(V)$ is Lipschitz continuous at $V_*$,
then
\begin{equation} \label{eq:etasup2}
\eta_{\sup} =
\limsup_{V\in \bbU^{n\times k} \atop \|\tan\Theta(V,V_*)\|_{\UI}\rightarrow 0}
\frac {\| \SL(V)\|_{\UI}} { \|\tan\Theta(V,V_*)\|_{\UI}}
\quad <\ \infty,
\end{equation}
where $\SL(V)$ is the unique solution of the Sylvester equation
defined in~\eqref{eq:xv}.

\item \label{i:thm:lnorm:a2}
If $H(V)$ is differentiable at $V_*$.  Then
\begin{equation}\label{eq:etainfty}
\eta_{\sup} = \vvvert {\OP}\vvvert_{\UI}\ge \eta_{\sup,\infty}=\rho({\OP}),
\end{equation}
where 
$\OP$ is the {\LOP} of the plain SCF defined in~\eqref{eq:ls},  
$\vvvert\OP \vvvert_{\UI}$ is the operator norm of $\OP$
    induced by the unitarily invariant norm $\|\cdot\|_{\UI}$, i.e.,
    $\vvvert\OP\vvvert_{\UI} :=   \sup_{Z\neq 0}
    \frac{\|\OP(Z)\|_{\UI}}{\|Z\|_{\UI}}.$

Consequently, the plain SCF \eqref{eq:pscf} is locally convergent to
$V_*$ with its asymptotic average convergence rate bounded by $\rho(\OP)$ if $\rho(\OP) < 1$, and
 locally divergent at $V_*$ if $\rho(\OP) > 1$.
\end{enumerate}
\end{theorem}

\begin{proof}
For item~\ref{i:thm:lnorm:a1},
by definition~\eqref{eq:bestlip} with $d(\cdot,\cdot):=\|\Theta(\cdot,\cdot)\|_{\UI}$,
we obtain
    \begin{equation}\label{eq:etaui}
    \eta_{\sup} =
    \limsup_{V_0\in \bbU^{n\times k} \atop  \|\Theta(V_0,V_*)\|_{\UI}\to 0}
    \frac{ \|\Theta(V_1,V_*)\|_{\UI}}{\|\Theta(V_0,V_*)\|_{\UI}} =
    \limsup_{V_0\in \bbU^{n\times k} \atop  \|\tan\Theta(V_0,V_*)\|_{\UI}\to 0}
    \frac{ \|\tan\Theta(V_1,V_*)\|_{\UI}}{\|\tan\Theta(V_0,V_*)\|_{\UI}},
    \end{equation}
    where the second equality is a consequence of~\eqref{eq:ts},
    together with $\Theta(V_1,V_*)\to 0$ as $\Theta(V_0,V_*)\to 0$ due to~\eqref{eq:onestep}.
    Then, a direct application of~\eqref{eq:onestep} leads
    to~\eqref{eq:etasup2}. 

For the boundedness of $\eta_{\sup} < \infty$,
by taking norms on the Sylvester solution~\eqref{eq:xv}
and exploiting the $2$-norm consistency property $\|AB\|_{\UI}\leq
\|A\|_2\|B\|_{\UI}$ of unitarily invariant norms, we have
\begin{equation}\label{eq:xvv}
   \|\SL(V)\|_{\UI}\leq
\delta_*^{-1}\|V_{*\bot}^{\HH}[H(V_*)-H(V)]V_*\|_{\UI}.
\end{equation}
On the other hand,
it follows from the Lipschitz continuity of $H(V)$ and~\eqref{eq:uninvar} that
\[
    \|H(V)-H(V_*)\|_{2}
    \leq \alpha \, \left(\| \tan\Theta(V,V_*) \|_{2} + \calo(\|
    \tan\Theta(V,V_*) \|_{2}^2)\right),
\]
for some constant $\alpha<\infty $.
Combining this with~\eqref{eq:xvv}
and~\eqref{eq:etasup2}, we conclude $\eta_{\sup} < \infty $.

For item~\ref{i:thm:lnorm:a2},
the inequality in \eqref{eq:etainfty} has already been established
in~\eqref{eq:eee},
and the formula of $\eta_{\sup}$ follows directly from~\eqref{eq:etasup2}
and the expansion~\eqref{eq:onestepdiff}.
It remains to find the expressions for $\eta_{\sup,\infty}$.


    Denote by $T_m=(V_{*\bot}^{\HH}V_m)(V_*^{\HH} V_m)^{-1}$ for
    $m=0,1,\dots$.
    It follows from~\Cref{lem:onestep} that
    \[
        T_m = \OP^m(T_0) + \lo(c_m\|T_0\|_{\UI}),
    \]
    where $\OP^m =\OP\circ\dots\circ\OP$
    represents the composition of the linear operator $\OP$ for $m$ times,
    and $c_m$ is a constant independent of $T_0$.
    Hence for any given $m$
\begin{align*} 
        \eta_{\sup,m} & =
        \limsup_{\|\Theta(V_0,V_*)\|_{\UI}\to 0}
        \left(\frac{\|\Theta(V_m,V_*)\|_{\UI}}{\|\Theta(V_0,V_*)\|_{\UI}}\right)^{1/m}
        =
        \limsup_{\|T_0\|_{\UI}\to 0}
        \left(\frac{\|T_m\|_{\UI}}{\|T_0\|_{\UI}}\right)^{1/m} \\ 
        & =
        \limsup_{T_0\to 0}
        \left(\frac{\|{\OP^m(T_0)}\|_{\UI}}{\|T_0\|_{\UI}}\right)^{1/m},
\end{align*} 
    where the second equation is due to~\eqref{eq:ts}, together with the continuity $T_m\to 0$ as $T_0\to 0$,
    implied by~\eqref{eq:onestep}.
    Since $\OP$ is a finite dimensional linear operator,
    we have that $\eta_{\sup,m} = (\vvvert \OP^m \vvvert_{\UI})^{1/m}$.
The expression for $\eta_{\sup,\infty}$ in~\eqref{eq:etainfty} 
is a consequence of
    Gelfand's formula, which says
    $\lim_{m\to \infty} \vvvert{\OP}^m\vvvert^{1/m}= \rho({\OP})$ for
    any operator norm $\vvvert\cdot\vvvert$ 
in a finite dimensional vector space (see, e.g.,~\cite[Thm 17.4]{Lax:2002}).
\end{proof}


In recent years,
a series of works, e.g.,~\cite{Cai:2018,Liu:2014,Yang:2009},
have been published to improve the upper bounds of the local contraction factor
$\eta_{\sup}$.
Those bounds were typically established for particular choices
of the metric $d(\cdot,\cdot)$ between subspaces and for a class of $H(V)$.
Let us revisit particularly the following convergence factor
of the plain SCF iteration presented recently in~\cite{Cai:2018}: 
\begin{align}\label{eq:czbl}
    \eta_{\czbl}:=
     \limsup_{V\in \bbU^{n\times k} \atop
    \|\sin\Theta(V,V_*)\|_{\UI}\rightarrow 0}
    \frac {\delta_*^{-1}\|V_{*\bot}^{\HH}[H(V_*)-H(V)]V_*\|_{\UI}} {\|
    \sin\Theta(V,V_*)\|_{\UI}}.
\end{align}
For a differentiable $H(V)$ with the expansion~\eqref{eq:diff},
$\eta_{\czbl}$ can be simplified as
\begin{equation}\label{eq:etac}
\eta_{\czbl} = \delta_*^{-1} \cdot \vvvert\OP_{\czbl} \vvvert_{\UI},
\end{equation}
where $\OP_{\czbl}:\bbC^{(n-k)\times k} \to \bbC^{(n-k)\times k}$
is an $\bbR$-linear operator:
\begin{equation}\label{eq:OP-czbl}
\OP_{\czbl}(Z) = V_{*\bot}^{\HH}\,  \mathbf D H(V_*)[V_{*\bot} Z]\,  V_*.
\end{equation}
The convergence factor $\eta_{\czbl}$ in~\eqref{eq:czbl} has significantly
improved several previously established results
in \cite{Liu:2014,Yang:2009}.
However, it follows from the characterization of $\eta_{\sup}$ in~\eqref{eq:etasup2}
and the bound of $\SL(V)$ in~\eqref{eq:xvv} that
\begin{align}\label{eq:ub-czbl}
\eta_{\sup}  \leq \eta_{\czbl}.
\end{align}
Therefore, the quantity $\eta_{\czbl}$ is an upper bound of
$\eta_{\sup}$, and could
substantially underestimate the convergence rate of SCF in practice, 
see numerical examples in~\Cref{sec:examples}.

We have already seen from~\Cref{lem:2sup} that $\eta_{\sup,\infty}$ is
an optimal convergence factor for SCF and $\eta_{\sup,\infty}\leq \eta_{\sup}$.
To see how large  the gap between
$\eta_{\sup,\infty}$ and $\eta_{\sup}$ in~\eqref{eq:etainfty} might get, we consider in particular the
local contraction factor $\eta_{\sup}$ in the commonly used Frobenius norm:
\begin{align}
\vvvert\OP\vvvert_{\F}
  &:=\sup_{Z\neq 0} \frac{\|\OP(Z)\|_{\F}}{\|Z\|_{\F}}
    = \sup_{Z\neq 0 } \frac{\langle \OP(Z),\OP(Z)\rangle^{1/2} }{\langle Z,Z\rangle^{1/2} } \nonumber\\
  &= \sup_{Z\neq 0 } \frac{\langle Z,\OP^*\circ\OP(Z)\rangle^{1/2} }{\langle Z,Z\rangle^{1/2} }
   = \left(\lambda_{\max} (\OP^*\circ\OP)\right)^{1/2},\label{eq:lam}
\end{align}
where $\langle X,Y\rangle = \Re (\tr(X^{\HH}Y))$ denotes the inner product
for $\bbC^{(n-k)\times k}(\bbR)$,
and $\OP^{*}$ is the adjoint of $\OP$.
It follows from~\eqref{eq:etainfty} that
\begin{equation}\label{eq:eta-sup-F}
    \eta_{\sup}= |\lambda_{\max} (\OP^*\circ\OP)|^{1/2}
\geq |\rho(\OP)| = \eta_{\sup,\infty}.
\end{equation}
By the standard matrix analysis, the equality in \eqref{eq:eta-sup-F} 
holds if $\OP$ is a {\em normal} linear operator on $\bbC^N(\bbR)$,
and the gap between the two numbers can be arbitrarily large when $\OP$ is far from normal.
For practical NEPv, such as the ones in Section~\ref{sec:examples}, 
we have observed that $\OP$ is usually
a slightly non-normal operator, causing a small gap between 
the two contraction factors.
Recall that $\eta_{\sup}$ is dependent of the metrics $d$.
Another possibility for the equality in~\eqref{eq:eee} to hold is through
a particular choice of metric.
Unfortunately, the optimal metric for $\eta_{\sup}$ is generally hard to know.

Finally, we comment on another recent work~\cite{Upadhyaya:2018}
on the local convergence analysis of SCF using the spectral radius.
In~\cite{Upadhyaya:2018},
SCF is viewed as a fixed-point iteration $P_{k+1} = \psi(P_k)$
in the density matrix $P_k = V_kV_k^{\HH}\in\bbC^{n\times n}$,
rather than in $V_k$ directly.
The authors showed that the fixed-point mapping $\psi(P)$ has a closed-form
Jacobian supermatrix $J$, assuming $H(V)$ is a linear function in
$P=VV^{\HH}$. So the spectral radius of $J$ also provides 
a convergence criterion.
Since $P$ has $p=(n+1)n/2$ free variables, the corresponding
supermatrix $J$ is
of size $p$-by-$p$.
This is in contrast to the $\bbR$-linear operator
$\OP$~\eqref{eq:ls} in tangent-angle matrices, which is only of 
size $q$-by-$q$ with $q=2(n-k)k = \calo ({p}^{1/2})$.
In addition to the reduced size, the use of linear operator, rather than
a supermatrix, allows for more convenient computation of the spectral radius in
practice, as will be discussed in~\Cref{sec:setup}.
Furthermore, $\OP$ is also easier to work with theoretically and numerically,
thanks to its simplicity in formulation and more explicit
dependencies on key variables, such as derivatives and eigenvalue gaps.
In the next section, we will show how to apply the spectral radius
$\rho(\OP)$ to analyze the so-called {\em level-shifting scheme\/}
for stabilizing and accelerating the plain SCF iteration.

\section{Level-shifted SCF iteration}\label{sec:ls}
In the previous section, we have discussed that
if the spectral radius $\rho(\OP) >1$ (or more generally $\eta_{\sup,\infty}>1$ in the case when $H(V)$ is simply just continuous),
then the plain SCF~\eqref{eq:pscf} is locally divergent at $V_*$.
However, even if $\rho(\OP) <1$, the process is prone
to slow convergence or oscillation before reaching local convergence.
To address those issues, the plain SCF may be applied in practice
with some stabilizing schemes to help with convergence.
Among the most popular choices is the level-shifting strategy
initially developed 
in computation chemistry~\cite{Saunders:1973,Thogersen:2004,Yang:2007}.
In this section, we discuss why such a scheme can work through the lens of
spectral radius when $H(V)$ is differentiable.

\subsection{Level-shifted SCF iteration}
The level-shifting scheme
modifies the plain SCF~\eqref{eq:pscf} with a parameter $\sigma$
as follows:
\begin{equation}\label{eq:levelshift}
    [H(V_i)-\sigma \, V_iV_i^{\HH}] V_{i+1} = V_{i+1}\Lambda_{i+1},
    \quad\text{for}\quad i = 0,1,2,\dots,
\end{equation}
where $V_{i+1}$ is an orthonormal basis matrix of the invariant subspace associated with the $k$ smallest eigenvalues of
the matrix $H(V_i)-\sigma\, V_iV_i^{\HH}$.
It can be viewed simply as the plain SCF~\eqref{eq:pscf}
applied to the level-shifted NEPv
\begin{equation}\label{eq:shifted}
H_{\sigma}(V) V = V \Lambda
\quad\text{with\quad $H_{\sigma}(V):=  H(V)-\sigma VV^{\HH}$}.
\end{equation}
Note that $H_{\sigma}(V)$ is again unitarily invariant as in~\eqref{eq:univar}.
The level-shifting transformation does not alter the solutions
of the original NEPv~\eqref{eq:nepv}, but shifts related eigenvalues of $H(V)$ by $\sigma$:
\[
H(V) V = V \Lambda
\qquad\Longleftrightarrow\qquad
H_{\sigma} (V) V = V (\Lambda-\sigma I_k).
\]
Hence if $(V_*,\Lambda_*)$ is a solution of 
the original NEPv~\eqref{eq:nepv},
then $(V_*,\Lambda_* -\sigma I_k)$ will solve 
the level-shifted NEPv
\begin{equation} \label{eq:nepvls} 
H_{\sigma}(V)V = V\Lambda.
\end{equation}
In the following discussion, we assume the parameter $\sigma$ is a constant
for convenience. In practice, it can change  iteration-by-iteration.

One direct consequence of the level-shifting transformation
is that it enlarges the eigenvalue gap at the solution
$V_*$.  By the eigen-decomposition~\eqref{eq:eigdec}, 
we obtain
\begin{equation}\label{eq:eighsigma}
     H_{\sigma} (V_*) \,  [V_*,V_{*\bot}]
     = [V_*,V_{*\bot}]\,
    \begin{bmatrix}
        \Lambda_*- \sigma I_k & \\
                              & \Lambda_{*\bot}
    \end{bmatrix}.
\end{equation}
Recall that  $\Lambda_*=\diag (\lambda_1,\dots,\lambda_k)$
and $\Lambda_{*\bot}=\diag (\lambda_{k+1},\dots,\lambda_n)$
consist of the ordered eigenvalues of $H(V_*)$ as in~\eqref{eq:eigdec}.
Therefore, the gap between the $k$th and $(k+1)$st eigenvalue
of $H_{\sigma}(V_*)$ becomes
\begin{equation}
    \delta_{\sigma*}: = \lambda_{k+1} - (\lambda_k-\sigma) = \delta_{*} +
    \sigma,
\end{equation}
where $\delta_*$ denotes the eigenvalue gap~\eqref{eq:gap} of
the original NEPv~\eqref{eq:nepv} at $V_*$.
So the level-shifted NEPv~\eqref{eq:nepvls} always has a larger
eigenvalue gap  $\delta_{\sigma*}$ if $\sigma > 0$.

It is well-known that
the larger the eigenvalue gap between the desired eigenvalues
and the rest ones, the easier and more robust it will become to
compute the desired eigenvalues and the associated eigenspace~\cite{Davis:1970,Parlett:1998,Stewart:1990}.
Therefore, it is desirable to have
a large eigenvalue gap $\delta_{\sigma *}$
for the sequence of matrix eigenvalue problems
in the SCF iteration~\eqref{eq:levelshift}, but on the other hand too large a $\sigma$ negatively affects
the local convergence rate of SCF as numerical evidences suggest.
Presently, there are heuristic schemes to choose the level-shift
parameter $\sigma$ in practice, e.g., see~\cite{Yang:2007}.
However, those heuristics cannot explain how
the level-shifting parameter $\sigma$ is
directly affecting the convergence behavior of  SCF~\eqref{eq:levelshift}
for  NEPv~\eqref{eq:shifted}.

We should mention that the conventional restriction of $\sigma>0$ 
for the level-shift parameter~\cite{Saunders:1973, Thogersen:2004,Yang:2007}
is not necessary.
We can see from the eigen-decomposition~\eqref{eq:eighsigma} that,
provided $\sigma \in (-\delta_*,+\infty)$, the eigenvectors $V_*$
always correspond to the $k$ simallest eigenvalues of $H_{\sigma}(V_*)$.

\subsection{Spectral radius for level-shifted {\LOP}}
In what follows, we investigate the local convergence behavior
of the level-shifting scheme by examining
the spectral radius $\rho(\OP_{\sigma})$
for the {\LOP} $\OP_{\sigma}$ of the level-shifted SCF~\eqref{eq:levelshift}.
We will focus on a class of NEPv where certain conditions on the
derivatives of $H(V)$ will apply.
Those conditions hold for NEPv arising in optimization
problems with orthogonality constraints,
as is usually the case for most practical NEPv.

\subsubsection{NEPv from optimization with orthogonality constraints}
~Let $H(V)$ be differentiable.
Define the $\bbR$-linear operator
$\scrQ\colon \bbC^{(n-k)\times k} \to  \bbC^{(n-k)\times k}$
by
\begin{equation}\label{eq:qs}
    \scrQ(Z) :=
    V_{*\bot}^{\HH} \,  \mbox{\bf D} H(V_*)[V_{*\bot} Z]\,  V_*    + \Lambda_{*\bot} Z - Z \Lambda_*.
\end{equation}
We call $\scrQ$ a {\em \QOP} of
NEPv~\eqref{eq:nepv}, and make the following assumption.

\begin{assumption}\label{ass:qs}
The linear operator $\scrQ$ is self-adjoint and positive definite 
with respect to the standard inner product on $\bbC^{(n-k)\times k}$, 
i.e.,
$$
\Re(\tr (Z^{\HH}\scrQ(Z))) = \Re (\tr ([\scrQ(Z)]^{\HH}Z))
\quad \mbox{and} \quad
\mbox{$\Re (\tr (Z^{\HH}\scrQ(Z))) > 0$ for all $Z\neq 0$}.
$$
\end{assumption}

To justify~\Cref{ass:qs}, let us take a quick review of a
class of NEPv arising from the following optimization problems 
with orthogonality constraints
\begin{equation}\label{eq:manopt}
\min_{V\in\bbC^{n\times k}} E(V) \quad \text{s.t.}\quad V^{\HH}V = I_k,
\end{equation}
where $E$ is some energy function satisfying $\nabla E(V) = H(V)V$
(see, e.g.,~\cite{Bao:2013,Yang:2007,Zhang:2014}).
We will make no assumption on the specific form of $E(\cdot)$ to be used.
For the constrained optimization problem~\eqref{eq:manopt}, the associated
Lagrangian function is given by
\[
    L(V):=  E(V) +
    \frac{1}{2}\tr \left(\Lambda^{\HH} (V^{\HH}V - I_k)\right),
\]
where $\Lambda=\Lambda^{\HH}$ is the $k$-by-$k$ matrix 
of Lagrange multipliers. 
We have suppressed $L$'s dependency on $\Lambda$
for notation simplicity.
The first order optimization condition
$\nabla_{V}  L(V) = H(V)V - V\Lambda=0$
leads immediately to  NEPv~\eqref{eq:nepv}. 

Because the target solution $V_*$ of interest is also
a minimizer of~\eqref{eq:manopt}, it needs to
satisfy certain second order condition as well.
Assuming $E(V)$ is also second order differentiable,
by straightforward derivation, 
the Hessian operator of $L(V)$ is given by
\[
    \nabla^2_V L(V_*)[X]  = H(V_*)X +
    \left(\mathbf D H(V_*)[X]\right) V_* - X\Lambda_*,
\]
where $X$ denotes the direction for the evaluation,
and $\mathbf D H(V_*)[\cdot]$ denotes the directional derivative of $H$
as defined in~\eqref{eq:dh}.
By the standard second-order optimization condition~\cite{Nocedal:2006},
this operator needs to be at least positive semi-definite
when restricted to $X=V_{*\bot}Z$ for all $Z\in\bbC^{(n-k)\times k}$,
namely, within the tangent space of the feasible set $V^{\HH}V=I_k$ at $V_*$.
Such a condition is included in Assumption~\ref{ass:qs},
where we further assume the positive definiteness of $\scrQ(\cdot)$.

\subsubsection{Spectral radius of level-shifted {\LOP}}
We can immediately draw from \Cref{lem:onestep} and \Cref{thm:lnorm} 
a conclusion that
the local convergence behavior of
the level-shifted SCF~\eqref{eq:levelshift} is characterized by
the {\LOP} corresponding to the level-shifted
NEPv~\eqref{eq:shifted}. To show the dependency on $\sigma$, we denote
this {\LOP} as
\begin{equation}\label{eq:lsigma}
  \OP_{\sigma}(Z)  =  D_{\sigma}(V_*)\odot (V_{*\bot}^{\HH} \,
                \mbox{\bf D}H_{\sigma}(V_*)[V_{*\bot} Z]\,  V_*),
\end{equation}
where $D_{\sigma}(V_*)\in\bbR^{(n-k)\times k}$  has elements
$D_{\sigma}(V_*)_{(i,j)} = (\lambda_{k+i}(H(V_*)) - \lambda_j(H(V_*)) +
\sigma)^{-1}$.
A representation of $\OP_{\sigma}$ in terms of {\QOP} $\scrQ$ and 
a bound of the spectral radius of $\OP_{\sigma}$ are given 
in the following theorem.

\begin{theorem}\label{thm:rhosigma}
Suppose~\Cref{ass:gap,ass:qs}, and $\sigma \in(-\delta_*,+\infty)$.
The {\LOP} $\OP_{\sigma}(\cdot)$ of the level-shifted SCF~\eqref{eq:levelshift} 
for the level-shifted NEPv~\eqref{eq:nepvls} is given by
    \begin{equation}\label{eq:lsigma2}
        \OP_{\sigma}(\cdot) = D_{\sigma} (V_*)\odot \scrQ(\cdot ) - I_{\rm id},
    \end{equation}
    where $\scrQ$ is the {\QOP} defined in~\eqref{eq:qs}
    and $I_{\rm id}$ denotes the identity operator on the vector space $\bbC^{(n-k)\times k}(\bbR)$.
    Moreover, the spectral radius of $\OP_{\sigma}$ is bounded:   
    \begin{equation}\label{ineq:rhosigma}
        \rho(\OP_{\sigma}) \le
        \max\left\{ \left|\frac{\mu_{\max}}{\sigma + \delta_*} -1\right|,
        \left|\frac{\mu_{\min}}{\sigma +  s_*} -1\right| \right\},
    \end{equation}
where $\mu_{\max}\geq \mu_{\min}>0$ denote
the largest and smallest eigenvalues of the $\bbR$-linear operator $\scrQ$,
$\delta_*$ and  $s_*$ are the spectral gap and span, respectively, i.e., 
$$
\delta_* = \lambda_{k+1}(H(V_*)) - \lambda_k(H(V_*)) 
\quad \mbox{and} \quad 
s_* = \lambda_{n}(H(V_*)) - \lambda_1(H(V_*)).
$$
\end{theorem}
\begin{proof}
    By the definition of $H_{\sigma}(V)$ in~\eqref{eq:shifted} and the derivative
    operator~\eqref{eq:dh}, it holds that
    \[
        \mbox{\bf D} H_{\sigma}(V_*)[X] =
        \mbox{\bf D} H(V_*)[X]  - \sigma \mbox{\bf D} (V_*V_*^{\HH})[X]
        = \mbox{\bf D} H(V_*)[X]  - \sigma (V_*X^{\HH} + XV_*^{\HH}).
    \]
    Hence
    \begin{align}
            V_{*\bot}^{\HH} \,
        \mbox{\bf D} H_{\sigma}(V_*)[V_{*\bot} Z]\,  V_*
        &= V_{*\bot}^{\HH} \,  \mbox{\bf D} H(V_*)[V_{*\bot} Z]\,  V_*
        -\sigma\,   Z  \nonumber \\
        &= \scrQ(Z) + Z (\Lambda_* -\sigma I_k) - \Lambda_{*\bot} Z
        = \scrQ(Z) - Z\oslash D_{\sigma}(V_*), \label{eq:vs} 
    \end{align}
    where the second equation is by~\eqref{eq:qs},
    and `$\oslash$' denotes the elementwise division.
    Plug \eqref{eq:vs} into~\eqref{eq:lsigma} to obtain
    \begin{align*}
        {\OP}_{\sigma}(Z)
        = D_{\sigma} (V_*)\odot [\scrQ(Z) - Z\oslash D_{\sigma}(V_*) ]
         = D_{\sigma} (V_*)\odot \scrQ(Z)  - Z.
    \end{align*}
    This proves~\eqref{eq:lsigma2}.

    The vector space $\bbC^{(n-k)\times k}(\bbR)$ has a natural basis
    $\mathcal B := \{E_{ij},\, \imath E_{ij} \colon i=1,\dots,n-k,\ j = 1,\dots,
    k\}$, where the entries of $E_{ij}\in\bbR^{(n-k)\times k}$ are all zeros but 1 as its
    $(i,j)$th entry.
    Let $\mathbf L_{\sigma},\mathbf D_{\sigma}, \mathbf Q\in\bbR^{2N\times 2N}$
    be the matrix representations of the operators $\OP_{\sigma}(\cdot)$,
    $D_{\sigma}(V_*)\odot (\cdot)$, and $\scrQ(\cdot)$ with respect to the basis
    $\mathcal B$, respectively, where $N=(n-k)\times k$.
    It follows from~\eqref{eq:lsigma2} that
    \[
        \mathbf L_{\sigma}  = \mathbf D_{\sigma} \mathbf Q - I_{2N}.
    \]
    Observe that $\mathbf D_{\sigma}$ is a diagonal matrix consisting of elements
    of $D_{\sigma}$,
    and $\mathbf Q$ is symmetric positive definite due to~\Cref{ass:qs}.
    Hence the eigenvalues of $\mathbf D_{\sigma} \mathbf Q $ are all positive, and
    \begin{equation}\label{eq:rhosigma:pf-1}
    \rho(\OP_{\sigma})=\max\{|\lambda_{\max}(\mathbf D_{\sigma} \mathbf Q)-1|,|\lambda_{\min}(\mathbf D_{\sigma} \mathbf Q)-1|\}.
    \end{equation}
    Since the eigenvalues of $\mathbf D_{\sigma} \mathbf Q$ are the same as those of
    $\mathbf Q^{1/2}\mathbf D_{\sigma} \mathbf Q^{1/2}$ and
    $$
    \lambda_{\max}(\mathbf D_{\sigma}) \mathbf Q
    \succeq \mathbf Q^{1/2}\mathbf D_{\sigma} \mathbf Q^{1/2}\succeq \lambda_{\min}(\mathbf D_{\sigma}) \mathbf Q,
    $$
   we have $\lambda_{\max}(\mathbf D_{\sigma} \mathbf Q)\le \mu_{\max} /(\sigma+\delta_*)$ 
    and $\lambda_{\min}(\mathbf D_{\sigma} \mathbf Q)\ge\mu_{\min}/(\sigma+s_*)$.
    Inequality \eqref{ineq:rhosigma} is now a simple consequence of \eqref{eq:rhosigma:pf-1}.
\end{proof}

It follows immediately from~\Cref{thm:rhosigma} that 
\begin{equation*}
\rho(\OP_{\sigma}) < 1 \quad
\mbox{if}\quad 
0 <
\frac{\mu_{\min}}{\sigma + s_*}\leq 
\frac{\mu_{\max}}{\sigma + \delta_*}\le 2, 
\end{equation*}
or equivalently, 
\begin{equation}\label{eq:lsbnd}
\rho(\OP_{\sigma}) < 1 \quad
\mbox{if}\quad \sigma\geq \frac{\mu_{\max}}{2} -\delta_*. 
\end{equation}
Hence for a sufficiently large $\sigma$,
the level-shifted SCF  is locally convergent!
On the other hand,
it also reveals that
$\rho_{\sigma}(\OP)\to 1$ as $\sigma\to +\infty$, implying the slow convergence of
the level-shifted SCF.
Further, if 
good estimates to $\mu_{\min}$, $\mu_{\max}$, $\delta_*$, and $s_*$ 
are available,
we may find a decent $\sigma$ by minimizing the upper bound 
in~\eqref{ineq:rhosigma} as follows: 
the minimizer is achieved when the two terms 
in the right-hand side of~\eqref{ineq:rhosigma}
coincide, which can happen only if 
$$
\frac{\mu_{\max}}{\sigma + \delta_*} -1=1-\frac{\mu_{\min}}{\sigma +  s_*},
$$
due to $\sigma\in ( -\delta_*, +\infty)$. 
This equation has a unique solution $\sigma_*\in (-\delta_*,+\infty)$.
Hence the operator $\OP_{\sigma}$ and 
its spectral radius provide us the understanding of level-shifting strategy 
and an approach to seek an optimal choice of the level-shifting 
parameter $\sigma$, see numerical examples in Section~\ref{sec:examples}.



To end this section, we note that 
the results in this section is consistent with,
and also complements, the convergence analysis
of the level-shifted methods applied to Hatree-Fock
equations~\cite{Cances:2000}.
Using optimization approaches, the authors 
showed that a sufficiently large shift $\sigma$ 
can lead to global convergence.
The condition~\eqref{eq:lsbnd}, on the other hand, provided
a closed-form lower bound on the size of $\sigma$ needed to achieve
local convergence. The bound of~\eqref{eq:lsbnd} involves the 
exact solution $V_*$ and is mostly of theoretical interest. 
For particular applications, it may be possible to have an 
{\em a-priori} estimate of $V_*$, as demonstrated in the examples 
in the next section.

\section{Numerical examples}\label{sec:examples}
In this section, we provide numerical examples to demonstrate
the sharpness and optimality 
of the convergence rate estimates presented in the
previous sections. 
Specifically, the purpose of the examples
is two-fold: Firstly, to illustrate how these convergence results are
manifest in practice, where various convergence rate estimates are compared
and their sharpness in estimating the actual convergence rate is demonstrated;
Secondly, to investigate and gain insight into the influence of the
level-shifting parameter $\sigma$ on the convergence rate 
of SCF \eqref{eq:levelshift}.

\subsection{Experiment setup}\label{sec:setup}

We will perform two case studies, one is a discrete Kohn-Sham equation 
with real coefficient matrices $H(V)$, and the other from a discrete 
Gross-Pitaevskii equation with complex matrices.

All our experiments are implemented and conducted in MATLAB 2019.
In each simulation, the ``exact'' solution $V_*$ is computed by 
the plain SCF \eqref{eq:pscf}, when it is convergent,
to achieve a residual tolerance $\|H(V_*)V_* -V_*\Lambda_*\|_2\leq 10^{-14}$.
When the plain SCF failed to converge, $V_*$
is computed by the level-shifted SCF \eqref{eq:levelshift} 
with a properly chosen shift $\sigma$.

The convergence rate estimates to be investigated include:
\begin{enumerate}[i)] 
\item $\eta_{\czbl}$ by~\cite{Cai:2018}, computed
        as~\eqref{eq:etac} in the Frobenius norm,
    \item $\eta_{\sup} = \vvvert \OP\vvvert_{\F}$ in~\eqref{eq:etainfty}
      in the Frobenius norm, and
\item  $\eta_{\sup,\infty}=\rho(\OP)$ in~\eqref{eq:etainfty}.
\end{enumerate}
These convergence rate estimates will be compared with the {\em observed convergence rate} of SCF,
estimated from the convergence history of the SCF iteration 
by the least squares approximation on the last few iterations.


\paragraph{Evaluation of $\eta_{\sup,\infty} (=\rho(\OP))$}
Despite a matrix representation $\mathbf L$ is involved in the
definition~\eqref{eq:sprad}, 
its explicit formulation is not needed for computing
$\rho(\OP)$.
Recall that
$\OP\colon \bbC^{p\times k}\to \bbC^{p\times k}$ is an
$\bbR$-linear operator.
By viewing a complex matrix $X = X_r + \imath X_i\in\bbC^{p\times k}$
as a pair of real matrices $(X_r,X_i)$  consisting of the real
and imaginary parts,
we express $\OP$ 
as a linear operator
$\what{\OP}\colon
\bbR^{p\times k}\times\bbR^{p\times k} \to \bbR^{p\times
k}\times\bbR^{p\times k}$,
\begin{equation}
    \what{\OP}(X_r,X_i) = \left(\Re(\OP(X)), \Im(\OP(X))\right).
\end{equation}
The input (as well as the output) matrix pair $(X_r,X_i)$
can be regarded as a real ``vector'' of length-$2N$.
The largest eigenvalue in magnitude of the linear
operator $\what{\OP}$ can be computed conveniently by
MATLAB \texttt{eigs} function as follows: 
\begin{verbatim}
v2m = @(x) reshape(x(1:N)+1i*x(N+1:end), p, []); % real vec x -> mat X
m2v = @(X) [real(X(:)); imag(X(:))];             % mat X -> real vec x
hatL = @(x) m2v(L(v2m(x))));                     % operator hat L
lam_max = eigs(hatL, 2*N, 1);                    % largest eigenval.
\end{verbatim}

\paragraph{Evaluation of $\eta_{\sup}$ and $\eta_{\czbl}$}
The induced norm $\vvvert\OP\vvvert_{\F}$ in~\eqref{eq:lam} is defined as the square root of the
largest eigenvalue of $\OP^* \circ \OP$, which is also an $\bbR$-linear
operator.
We can use exactly the same approach  above to obtain
$\lambda_{\max}(\OP^* \circ \OP)$.
Since the operator $\OP^* \circ \OP$ is self-adjoint,
the largest eigenvalue is always a real number.
In analogy, for $\eta_{\czbl}$ in~\eqref{eq:etac},  $\vvvert \OP_{\czbl}\vvvert_{\F}$
 can be computed as 
the square root of ${\lambda_{\max}(\OP_{\czbl}^* \circ \OP_{\czbl})}$.

\subsection{Single particle Hamiltonian}  
Let us consider an NEPv~\eqref{eq:nepv} with a real coefficient matrix-valued function
\begin{equation}\label{eq:hamilton}
H(V) = L + \alpha \,  \Diag (L^{-1} \diag (VV^{\T})),
\end{equation}
where tridiagonal matrix $L=\tridiag (-1,2,-1)\in\bbR^{n\times n}$ is a discrete 1D
Laplacian, $\alpha>0$ is a given parameter,
and $V\in\bbO^{n\times k}:=\{X\in\bbR^{n\times k}\,:\, X^{\T}X=I_k\}$.
$H(V)$  is known as the single-particle Hamiltonian
arising from discretizing an 1D Kohn-Sham equation in electronic structure
calculations, and has become a standard testing problem for investigating
the convergence of SCF due to its simplicity, 
see, e.g.,~\cite{Cai:2018,Liu:2014,Yang:2009,Zhao:2015}.
$H(V)$ is differentiable.
By a straightforward calculation, the directional derivative operator $\mbox{\bf D}H(V_*)$ defined in~\eqref{eq:dh}
is given by
\[
\mathbf D H(V)[X] = 2\alpha\,  \Diag (L^{-1} \diag (XV^{\T})),
\]
which is linear in $X$.

The {\LOP} $\OP$ in~\eqref{eq:ls} 
of the plain SCF \eqref{eq:pscf} is given by
\begin{equation}\label{eq:lreal}
\OP(Z) = 2\alpha \,  D(V_*)\odot \bigg(V_{*\bot}^{\T}\,
\Diag (L^{-1}\diag (V_{*\bot} Z V_*^{\T}))\,  V_* \bigg).
\end{equation}
The adjoint operator $\OP^*$ 
is given by
\begin{equation}\label{eq:lsreal}
    \OP^*(Y) = 2\alpha \,  V_{*\bot}^{\T}\,
    \Diag \left(L^{-\T}\diag \big(V_{*\bot}( D(V_*)\odot
    Y)V_*^{\T}\big)\right)\,  V_* ,
\end{equation}
see Appendix~\ref{app:adj} for the derivation. 

The {\LOP}
$\OP_{\sigma}$ \eqref{eq:lsigma} 
of the level-shifted SCF \eqref{eq:levelshift} 
is given by
\begin{equation} 
\OP_{\sigma}(Z) 
= D_{\sigma} (V_*) \odot \scrQ(Z) - I_{\rm id},   
\end{equation}
where $\scrQ$ is the {\QOP} $\scrQ$ defined in~\eqref{eq:qs} is 
given by 
\begin{equation}
    \scrQ(Z) = 2\alpha\,  V_{*\bot}^{\T}\,  \Diag (L^{-1}\diag (V_{*\bot} Z
    V_*^{\T}))\,  V_*    + (\Lambda_{*\bot} Z- Z\Lambda_*).
\end{equation}
The largest eigenvalue $\mu_{\max}$ of $\scrQ$ can be bounded
as follows:  let $Z\in\bbR^{(n-k)\times k}$ be the 
corresponding eigenvector of
$\mu_{\max}$, then
\begin{align}
    \mu_{\max}
    = \frac{\|\scrQ(Z)\|_{\F}}{\|Z\|_{\F}}
    &\leq  2\alpha\,  \frac{\|\Diag (L^{-1}\diag (V_{*\bot} Z
    V_*^{\T}))\|_{\F}}{\|Z\|_{\F}} + s_*\notag \\
    & \leq 2\alpha \,  \|L^{-1}\|_2 + s_*
    \leq 3\alpha \,  \|L^{-1}\|_2 + 4,\notag
\end{align}
where $s_*$ is the spectral span  of $H(V_*)$, 
and for the last inequality we have used 
the inequalities $s_* \leq \lambda_n(H(V)) \leq
\|L\|_2 + \alpha \|L^{-1}\|_2$ due to~\eqref{eq:hamilton},
and $\|L\|_2 \leq 4$.

Recalling the lower bound in \eqref{eq:lsbnd}  for the level-shifting parameter
$\sigma$, we find 
\begin{equation}\label{eq:sbnd}
\sigma \geq \frac{3}{2}\,  \alpha\,  \|L^{-1}\|_2 + 2
\geq \frac{\mu_{\max}}{2} - \delta_*
\end{equation}
is sufficient to ensure local convergence of SCF \eqref{eq:levelshift}.
The first inequality provides an {\em a-priori} lower bound on the shift
$\sigma$.  In practice, this crude bound is a bit pessimistic though.
But it does reveal two key contributing factors
--- the parameter $\alpha$ and size $n$ of the
problem due to the fact that 
$\|L^{-1}\|_2 = 2^{-1}(1-\cos(\frac{\pi}{n+1}))^{-1} = \calo(n^2) $
    for the 1D Laplacian \cite[Lemma~6.1]{Demmel:1997}
--- that tend to negatively affect the size of shift.

\begin{example}\label{ex:single}
In this example, we compare the sharpness 
of the three convergence rate estimates of the plain SCF.
We take $n=10$ and $k=2$, and use different $\alpha$ 
ranging from $0$ to $1$ in the Hamiltonian~\eqref{eq:hamilton}.
For each run of SCF, the starting vectors are set to be 
the basis of the $k$ smallest
eigenvalues of $L$. The results are 
shown in Figure~\ref{fig:ex1}.
A few observations are summarized as follows: 

\begin{figure}[t]
\begin{center}
\includegraphics[width=0.49\textwidth]{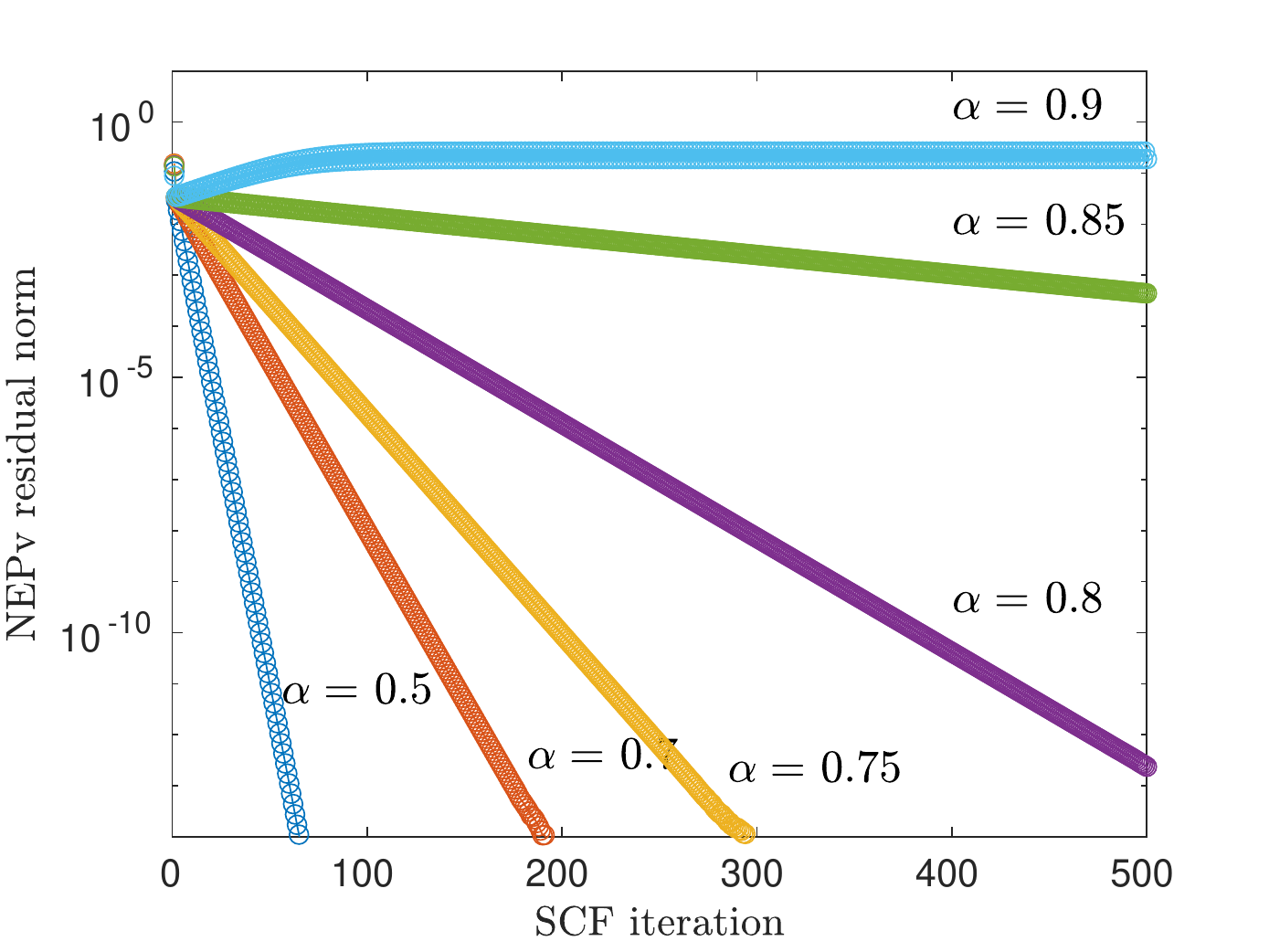}
\includegraphics[width=0.49\textwidth]{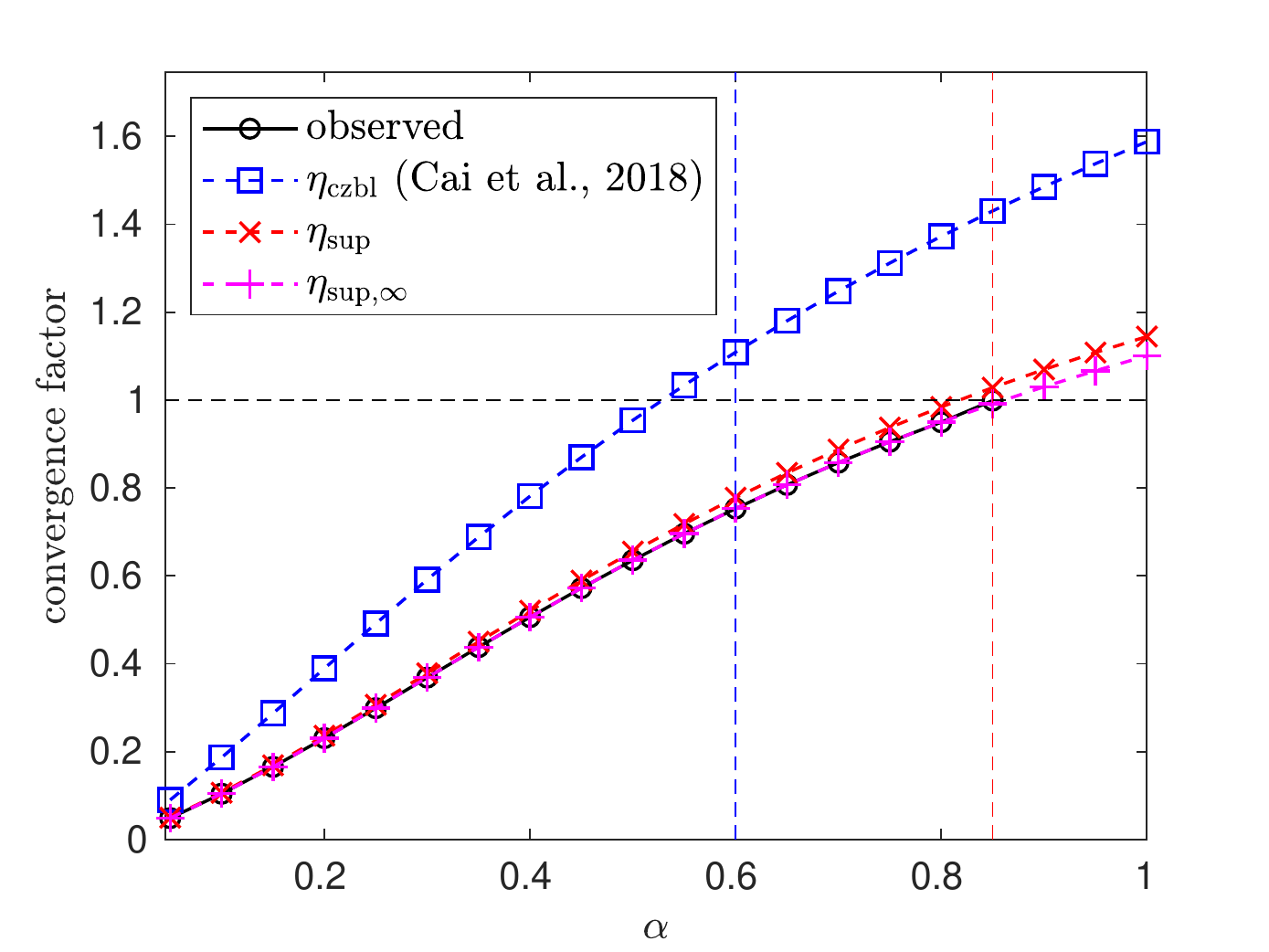}
\end{center}
\caption{
Example~\ref{ex:single}: 
convergence history
of residual norm $\|H(V_i)V_i-V_i\Lambda_i\|_2$ by the plain SCF \eqref{eq:pscf}
for selected $\alpha$ (left);  and
convergence rate estimates as $\alpha$ varies (right).
}\label{fig:ex1}
\end{figure}

\begin{enumerate}[(a)]
\item
For $\alpha=0$, the NEPv reduces to a standard eigenvalue problem
$LV=V\Lambda$, for which SCF converges in one iteration.
As $\alpha$ increases, SCF faces increasing challenges to converge.
In particular, for $\alpha$ larger than $0.85$, the plain SCF becomes divergent.
For those $\alpha$, the ``exact'' solutions $V_*$ used to calculate convergence factors are
computed by the level-shifted SCF.

\item
The asymptotic average contraction factor $\eta_{\sup,\infty} (= \rho(\OP))$ 
successfully 
predicts the convergence of SCF in all cases, and perfectly captures the 
convergence rate.
The factor $\eta_{\sup,\infty}$ yields excellent estimation after only
a small number of iterative steps, although 
strictly speaking, it is conclusive only as the iteration number
approaches infinity. 

\item
The contraction factor estimate $\eta_{\sup}$ is an
overestimate and usually provides a good prediction of local convergence.
It failed slightly at $\alpha=0.85$, where up to 10 digits: 
\[ 
\begin{array}{rclrrcl}  
 \mbox{observed} &  = & 0.9913931781, \,\, & \, \, \eta_{\sup,\infty} & = & 0.9913931591, \\
\eta_{\sup}      &  = & 1.028434776, \,\,   & \, \, \eta_{\czbl}      & = & 1.430511920.
\end{array} 
\] 
The gap between $\eta_{\sup,\infty}$ and
$\eta_{\sup}$ implies $\OP$ is a non-normal operator as discussed
in~\Cref{sec:sprd}.

\item
In comparison, the estimate 
$\eta_{\czbl}$ by~\cite{Cai:2018} is less precise. 
In particular, it fails to correctly indicate
the convergence of the plain SCF starting at $\alpha=0.55$,
which is in contrast to $\eta_{\sup}$ starting at $0.85$.
We mention that, for this same experiment,
it was illustrated in~\cite{Cai:2018} that
$\eta_{\czbl}<1$ for all $\alpha\leq 0.6$
(marked as dashed vertical line).
\end{enumerate}

\end{example}

\begin{example}\label{ex:single:ls}
In this example, we examine
{the convergence of the level-shifted SCF~\eqref{eq:levelshift} 
with respect to  the shift $\sigma$}.
The testing problem is the same as~\Cref{ex:single} 
but with a fixed $\alpha = 1$,
for which the plain SCF \eqref{eq:pscf} is divergent.
We apply the level-shifted SCF with various choices of $\sigma$ for the solution.
The convergence history and
the corresponding spectral radius of the operator $\OP_{\sigma}$
in~\eqref{eq:lsigma} is depicted in Figure~\ref{fig:ex1rhos}.

From the spectral radius plot on the right side of
Figure~\ref{fig:ex1rhos}, we observe that
$\rho(\OP_{\sigma})$ dropped quickly below $1$. 
The minimal value $\rho(\OP_{\sigma})\approx 0.33$
at $\sigma\approx 0.36$ and leads to rapid convergence of SCF
as shown in the left plot.
As $\sigma$ grows, $\rho(\OP_{\sigma})$ monotonically increases towards 1.
Such a behavior of $\rho(\OP_{\sigma})$ is consistent with the
bound obtained in~\Cref{thm:rhosigma}, governed by rational functions
in the form of $|1-a/(\sigma+b)|$ with $a,b>0$.

\begin{figure}[t] 
\begin{center}
\includegraphics[width=0.49\textwidth]{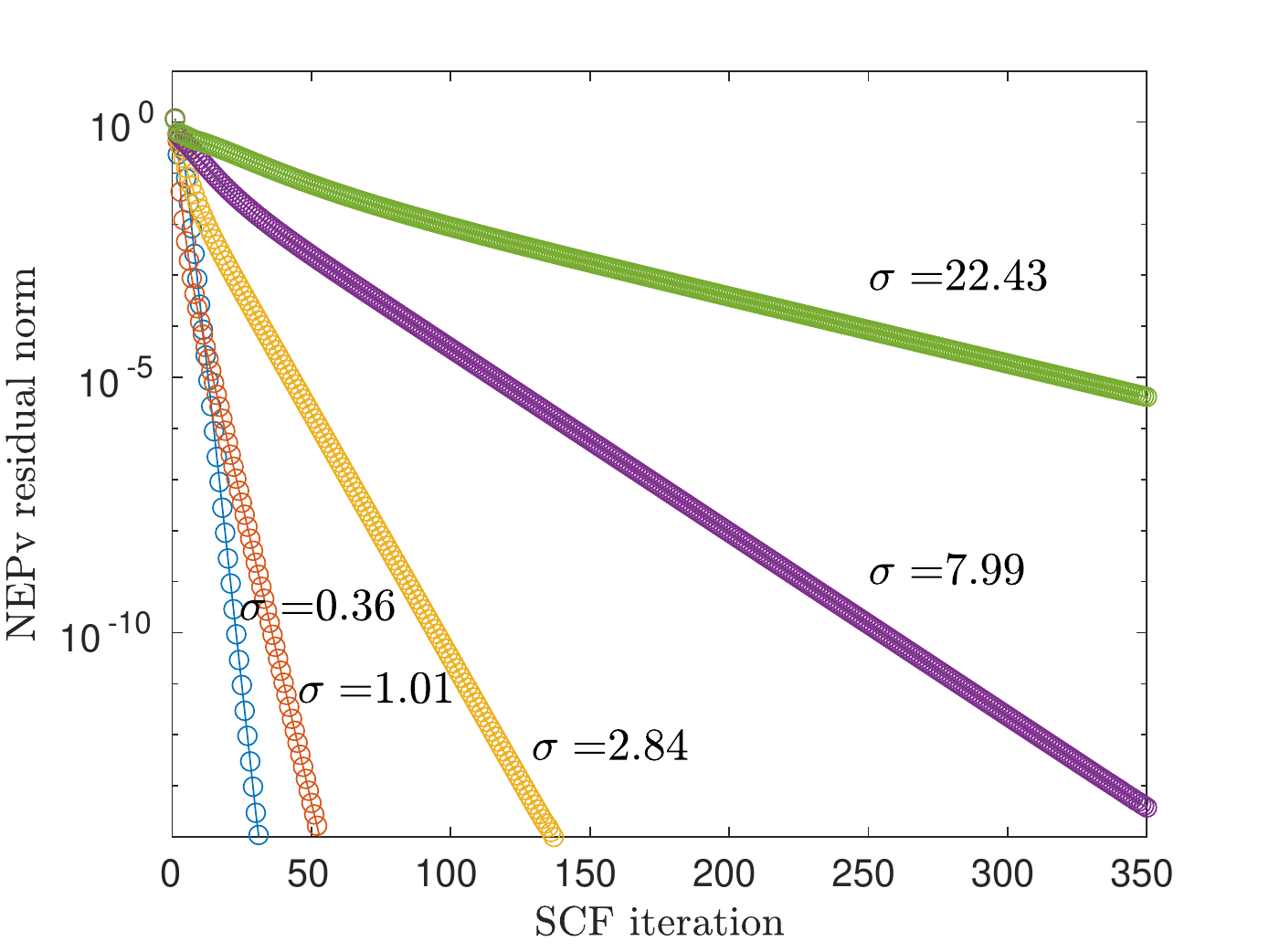}
\includegraphics[width=0.49\textwidth]{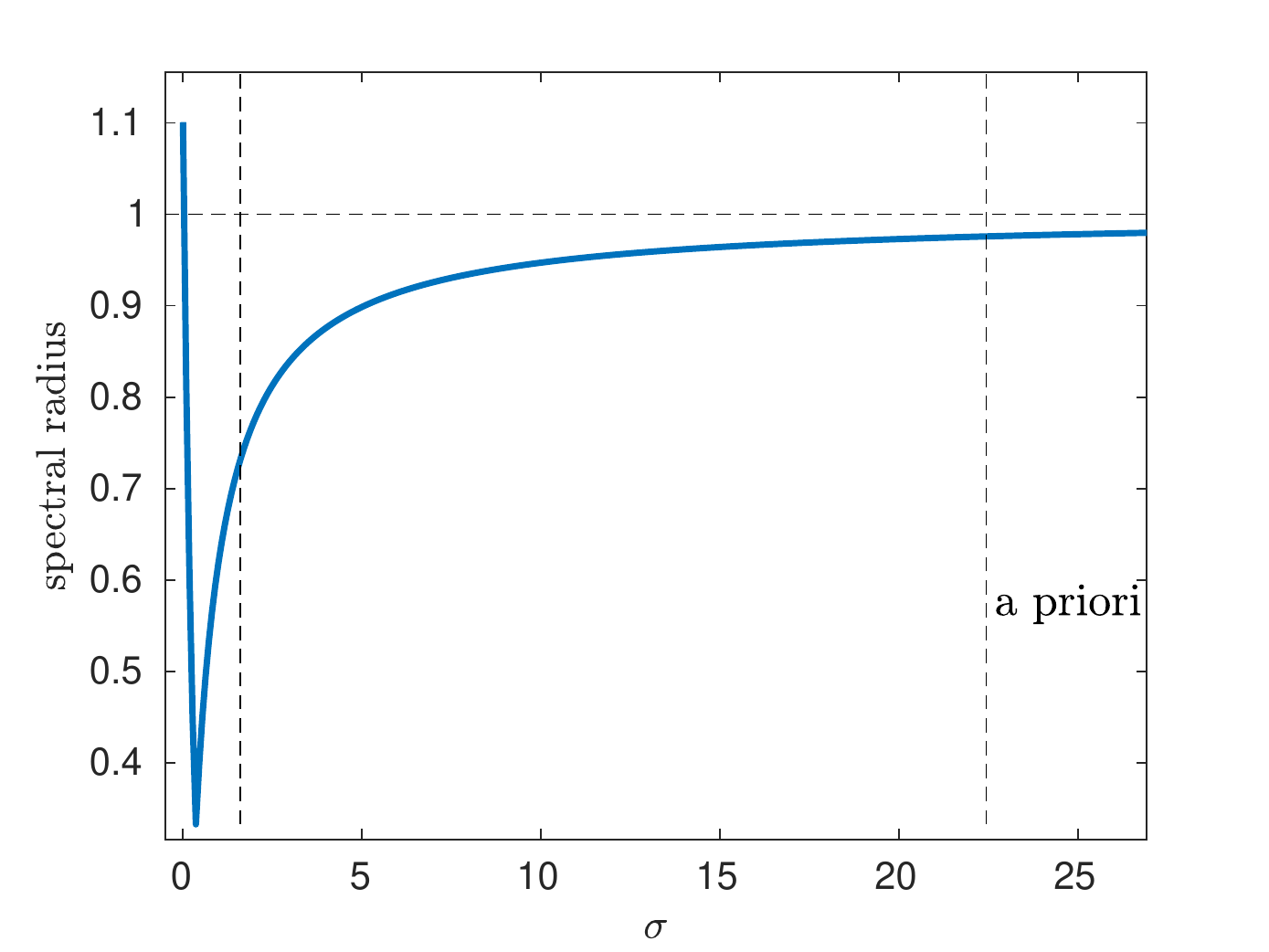}
\end{center}
\caption{Example~\ref{ex:single:ls}:
convergence history of residual norm $\|H(V_i)V_i-V_i\Lambda_i\|_2$
by the level-shifted SCF \eqref{eq:levelshift} with selected $\sigma$ (left);
spectral radius of $\rho(\OP_{\sigma})$ as
shift $\sigma$ varies (right), where the first vertical dash line 
is $\sigma=\frac{\mu_{\max}}{2} -\delta_*$
suggested by~\eqref{eq:lsbnd} and 
the second is {\em a-priori} 
$\sigma = \frac{3}{2}\,  \alpha\,  \|L^{-1}\|_2 + 2$ 
suggested by~\eqref{eq:sbnd},
and the optimal shift is $\sigma\approx 0.36$.
The $H(V)$ is given by~\eqref{eq:hamilton} with $\alpha=1$.
}\label{fig:ex1rhos}
\end{figure}

The sharp turning  of the curve of $\rho(\OP_{\sigma})$ reveals the
challenge in finding the optimal $\sigma$.
The values of spectral radius grows quickly as $\alpha$ moves away 
from the optimal shift. 
We note that both the theoretic lower bound in~\eqref{eq:lsbnd}
and {\em a-priori} estimate~\eqref{eq:sbnd} fall correctly
into the convergence region.
The {\em a-priori} bound provided a pessimistic estimate of $\sigma$, that
leads to a less satisfactory convergence rate of 
the level-shifted SCF~\eqref{eq:levelshift}.


\end{example}

\subsection{Gross--Pitaevskii equation}
In this experiment, we consider NEPv with complex 
coefficient matrices $H(V)$ given by
\begin{equation}\label{eq:gpe}
H(V) = A_f + \beta\,  \Diag (|V|)^2,
\end{equation}
where $A_f\in\bbC^{n\times n}$ is a  Hermitian matrix
and positive definite,
$\beta>0$ is a parameter, $V\in\bbC^n$ is a complex vector,
and $|\cdot|$ takes elementwise absolute value.
Such an NEPv arises from discretizing the Gross-Pitaevskii equation (GPE)
for modeling the physical phenomenon of 
Bose--Einstein condensation~\cite{Bao:2004,Jarlebring:2014,Jia:2016,Li:2020}.

The matrix $A_f$ in~\eqref{eq:gpe} is dependent of a potential
function $f$.  For illustration, we will discuss 
a model 2D GPE studied in~\cite{Jarlebring:2014},
where for a given potential function $f(x,y)$ over a two dimension domain
$[-\ell,\ell]\times[-\ell,\ell]$,
the corresponding matrix 
\begin{equation}\label{eq:af}
    A_f = \Diag(\widetilde f)-\frac{1}{2}M -\imath \omega M_{\phi},
\end{equation}
where
\[
\text{$\widetilde f = h^2\,
\left[f(x_1,y_1),\dots,f(x_N,y_1),f(x_1,y_2),\dots,f(x_N,y_2),\dots,f(x_N,y_N)\right]^{\T}\in\bbR^{N^2}$}
\]
with $\{x_i\}_{i=1}^N$ and  $\{y_i\}_{i=1}^N$ being interior points of the interval  $[-\ell ,\ell ]$ from
the $N+2$ equidistant discretization 
with spacing $h=\frac{2\ell}{N+1}$.
The matrices $M$, $M_{\phi}$ are given by
\[
M = D_{2,N}\otimes I + I\otimes D_{2,N},\,\,
M_{\phi} = h\, \Diag(y_1,\dots,y_N) \otimes D_N - D_N \otimes
    \left(h\, \Diag(x_1,\dots,x_N)\right),
\]
with $N \times N$ tridiagonal matrices 
$D_{N}=\tridiag (-\frac{1}{2},0,\frac{1}{2})$
and $D_{2,N}= \tridiag (1,-2,1)$.

Since $V$ is a vector, 
by definition~\eqref{eq:dh}
the directional derivative operator of $H(V)$ is given by
\[
\mathbf D H(V)[X] = 2\beta\, \Diag (\Re (\overline{V}\odot X)).
\]
The {\LOP} of the plain SCF $\OP:\bbC^{n-1}\to \bbC^n$ in~\eqref{eq:ls} is
\begin{equation}\label{eq:lcmpx}
\OP(Z) = 2\beta \,  D(V_*)\odot (V_{*\bot}^{\HH}\,
\Diag (
\Re (\overline{V}_*\odot (V_{*\bot} Z)))\,  V_*),
\end{equation}
and its adjoint operator $\OP^*$, with respect to the standard inner product in $\bbC^{(n-k)\times k}$  ($k=1$), i.e.,
$\langle \OP(Z), Y\rangle\equiv \Re (\tr(Y^{\HH}\OP(Z)))=\langle Z, \OP^*(Y)\rangle\equiv \Re (\tr([\OP^*(Y)]^{\HH} Z))$ for any $Y,\,Z\in\bbC^{(n-k)\times k}$,
is given by
\begin{equation}\label{eq:lscmpx}
   \OP^*(Y) = 2\beta \,  V_{*\bot}^{\HH}\,
   \left(\Re \left(\diag (V_{*\bot}( D(V_*)\odot
       Y)V_*^{\HH})\right)\odot V_*\right),
\end{equation}
see Appendix~\ref{app:adj} for the derivation. 

For the level-shifted SCF, the {\LOP}
$\OP_{\sigma}$ in \eqref{eq:lsigma} is given
by
\begin{equation} 
\OP_{\sigma}(Z) 
= D_{\sigma} (V_*) \odot \scrQ(Z) - I_{\rm id}, 
\end{equation}
where the {\QOP} $\scrQ(Z)$ 
is given by
\begin{equation}
    \scrQ(Z) =
    2\beta\,  V_{*\bot}^{\HH}\, \Diag (
     \Re (\overline{V_*}\odot (V_{*\bot}Z))\,  V_*
    + (\Lambda_{*\bot} Z- Z\Lambda_*).
\end{equation}
The largest eigenvalue $\mu_{\max}$ of $\scrQ$ can be bounded
as follows. Let $Z\in\bbC^{n-1}$ be the eigenvector associated with $\mu_{\max}$. Then
\begin{align}
    \mu_{\max}
    = \frac{\|\scrQ(Z)\|_{\F}}{\|Z\|_{\F}}
    &\leq  2\beta\,  \frac{\|
        \Diag ( \Re (\overline{V}\odot (V_{*\bot}Z))\|_{\F}}{\|Z\|_{\F}} + s_*\notag \\
    & \leq 2\beta + s_*
    \leq 3\beta + \|A_f\|_2,\notag
\end{align}
where $s_* = \lambda_n(H(V_*)) - \lambda_1(H(V_*))$ is the spectral
span, and for the last inequality we have used
the inequalities $s_* \leq \lambda_n(H(V_*)) \leq \beta + \|A_f\|_2 $
due to $H(V)$ in \eqref{eq:gpe} being positive definite.
Consequently,  the lower bound on $\sigma$ in~\eqref{eq:lsbnd} yields
\begin{equation}\label{eq:sigmagp}
    \sigma \geq \frac{1}{2}(3\beta+\|A_f\|_2)
\end{equation}
to ensure the local convergence of the level-shifted SCF.


\begin{example}\label{ex:gp}
In this example, we select the parameters
$\ell=1$, $\omega = 0.85$, and $N=10$ (hence $n=100$).
We use a radial harmonic potential
$f(x,y) = (x^2+y^2)/2$.
Various values of $\beta$ ranging from 
$0.5$ to $5$ have been tried.  The simulation results
    are shown in~Figure~\ref{fig:ex:gp}.

\begin{figure}[t]
\begin{center}
\includegraphics[width=0.49\textwidth]{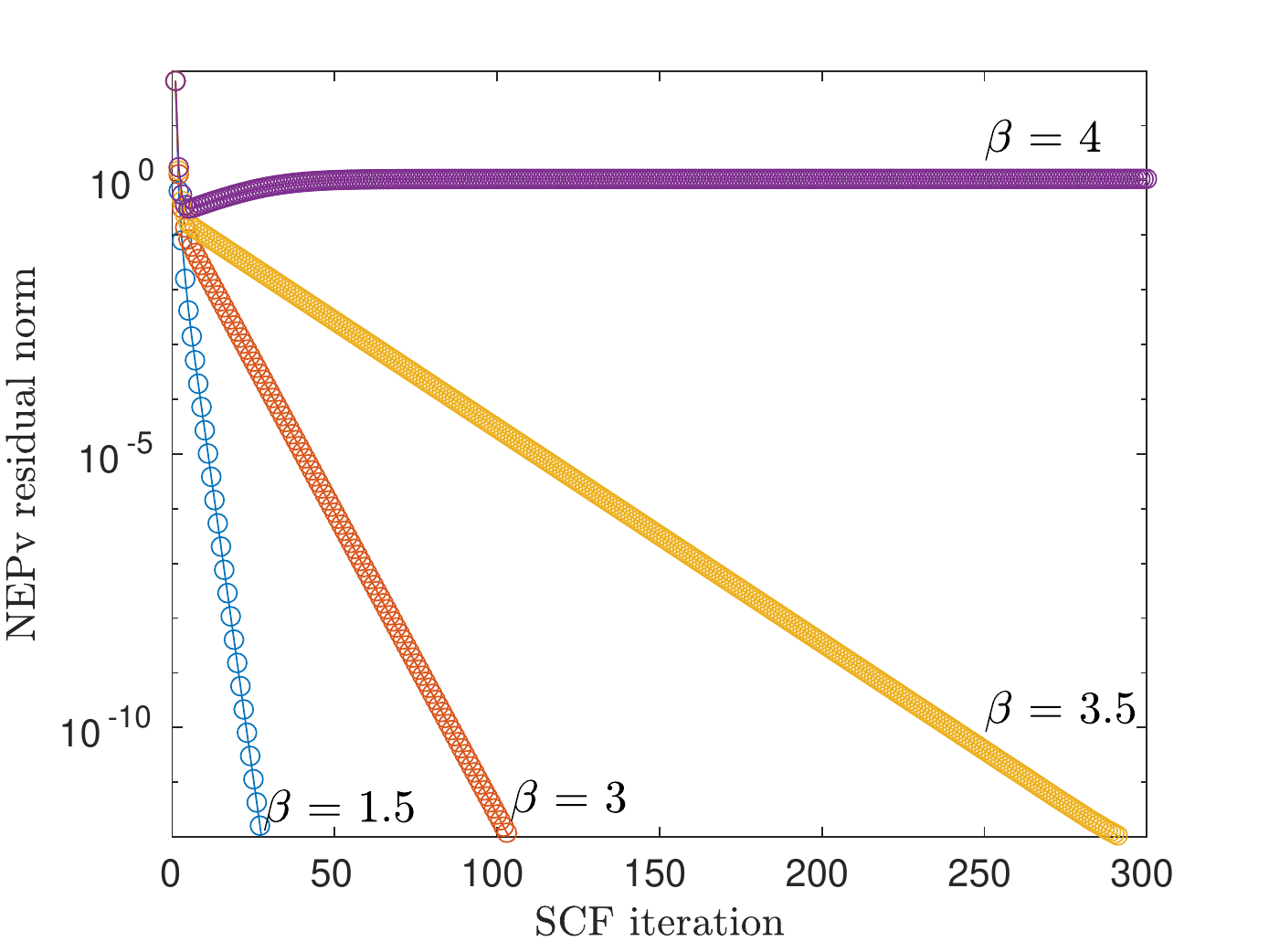}
\includegraphics[width=0.49\textwidth]{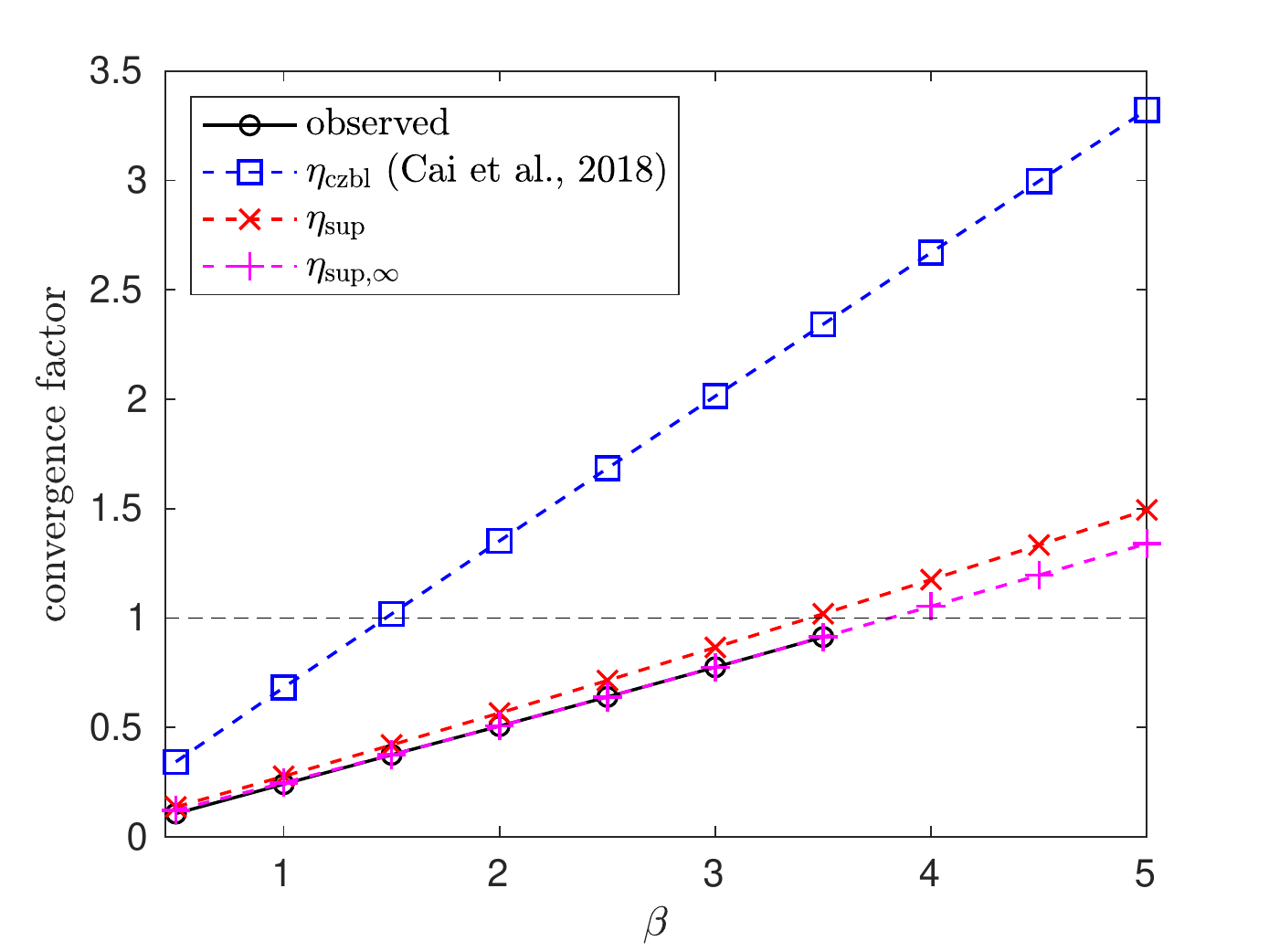}
\end{center}
\caption{Example~\ref{ex:gp}: (Left) Convergence history
of residual norm $\|H(V_i)V_i-V_i\Lambda_i\|_2$ by the plain SCF \eqref{eq:pscf}
for selected $\beta$; (Right)
Convergence rate estimates as $\beta$ varies.
}\label{fig:ex:gp}
\end{figure}

It is observed that
the plain SCF becomes slower and slower and eventually divergent 
as $\beta$ increases.
Again, the spectral radius $\rho(\OP_{\sigma})$ and $\eta_{\sup}$ can
well capture true convergence behavior.
In particular,  at $\beta = 3.5$, we find that up to 7 digits, 
\[ 
\mbox{observed}     = 0.9136140, \quad
\eta_{\sup,\infty}  = 0.9136173, \quad
\eta_{\sup}         = 1.019727, \quad
\eta_{\czbl}        = 2.342686
\] 
Again, we see the sharpness of the estimate $\eta_{\sup,\infty}$. 

The performance of the level-shifted SCF with respect to different
shifts $\sigma$ is shown in Figure~\ref{fig:ex:gp1:rhos},
where we observe a similar convergence behavior to 
Figure~\ref{fig:ex1rhos} of Example~\ref{ex:single:ls}
on the impact of the choice of shift $\sigma$.

\begin{figure}[ht] 
\begin{center}
\includegraphics[width=0.49\textwidth]{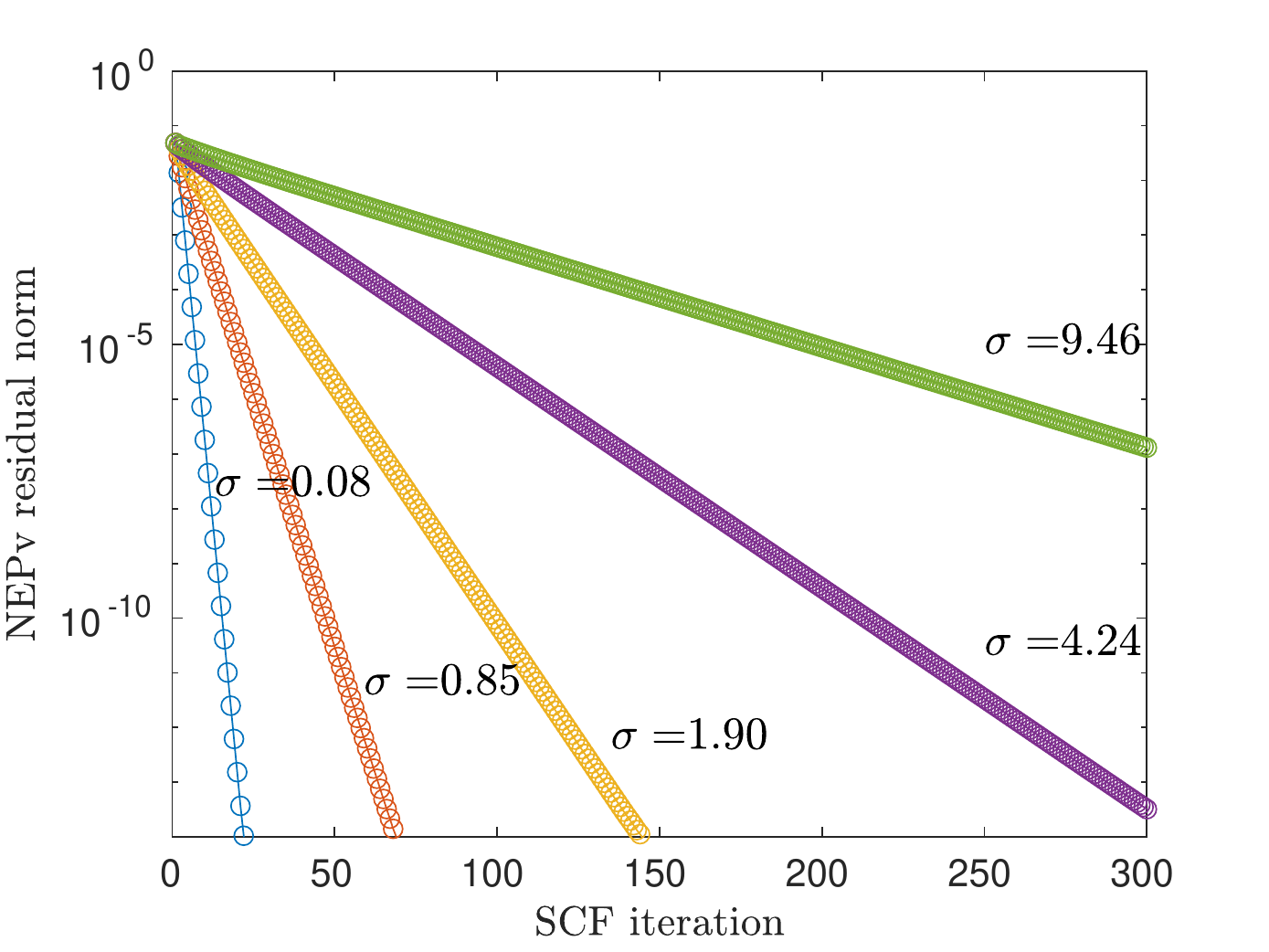}
\includegraphics[width=0.49\textwidth]{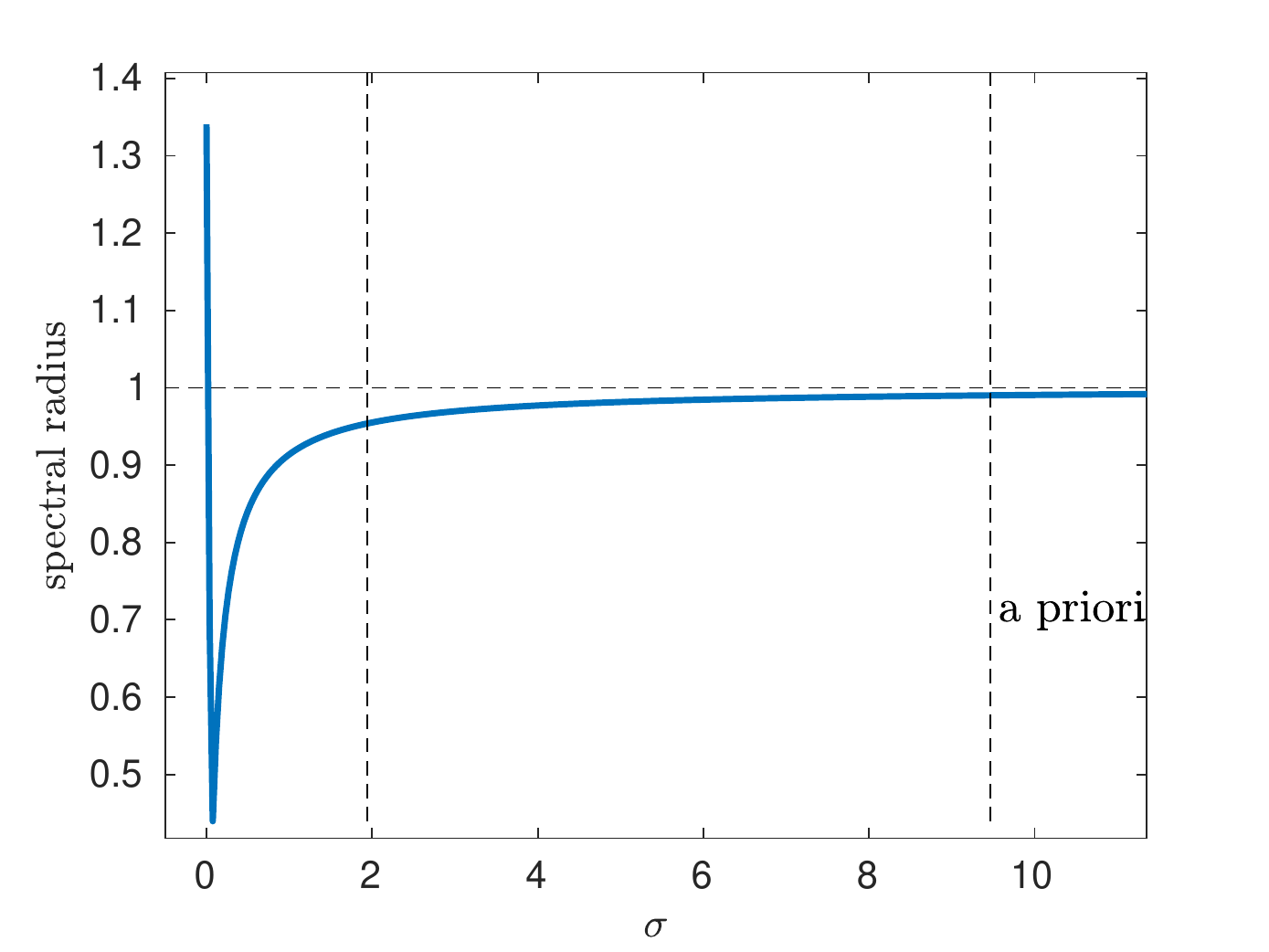}
\end{center}
\caption{Example~\ref{ex:gp}:
convergence history of residual norm $\|H(V_i)V_i-V_i\Lambda_i\|_2$
by the level-shifted SCF \eqref{eq:levelshift} with selected $\sigma$ (left);
spectral radius of $\rho(\OP_{\sigma})$ as
shift $\sigma$ varies (right), where the first vertical dash line 
is $\sigma=\frac{\mu_{\max}}{2} -\delta_*$
suggested by~\eqref{eq:lsbnd} and 
the second is {\em a-priori} $\sigma =
\frac{1}{2}(3\beta+\|A_f\|_2)$ suggested by~\eqref{eq:sigmagp},
and the optimal shift is $\sigma\approx 0.08$.
The $H(V)$ is given by~\eqref{eq:gpe} with $\beta=5$.
}\label{fig:ex:gp1:rhos}
\end{figure}

\end{example}

\begin{example}\label{ex:gp2}
This is a repeat of Example~\ref{ex:gp}, except using
a non-radical harmonic potential function  $f(x,y) = (x^2+100y^2)/2$.
The plots in Figure~\ref{fig:ex:gp:2} show
a slightly different performance of the plain SCF \eqref{eq:pscf} 
compared to the radical harmonic case of Example~\ref{ex:gp}.
The sharpness of the estimate $\eta_{\sup,\infty}$
on the local convergence rate can be seen at $\beta = 2.2$, where
up to  7 digits:
\[
\mbox{observed}     = 0.9652599, \quad  
\eta_{\sup,\infty}  = 0.9652614, \quad   
\eta_{\sup}         = 1.073434, \quad   
\eta_{\czbl}        = 2.043247 
\] 
The performance of the level-shifted SCF is depicted 
in Figure~\ref{fig:ex:gp2:rhos}.
Again we observe a similar convergence behavior 
to Example~\ref{ex:gp} with repect to the choice of shift $\sigma$. 

\begin{figure}[ht]
\begin{center}
\includegraphics[width=0.49\textwidth]{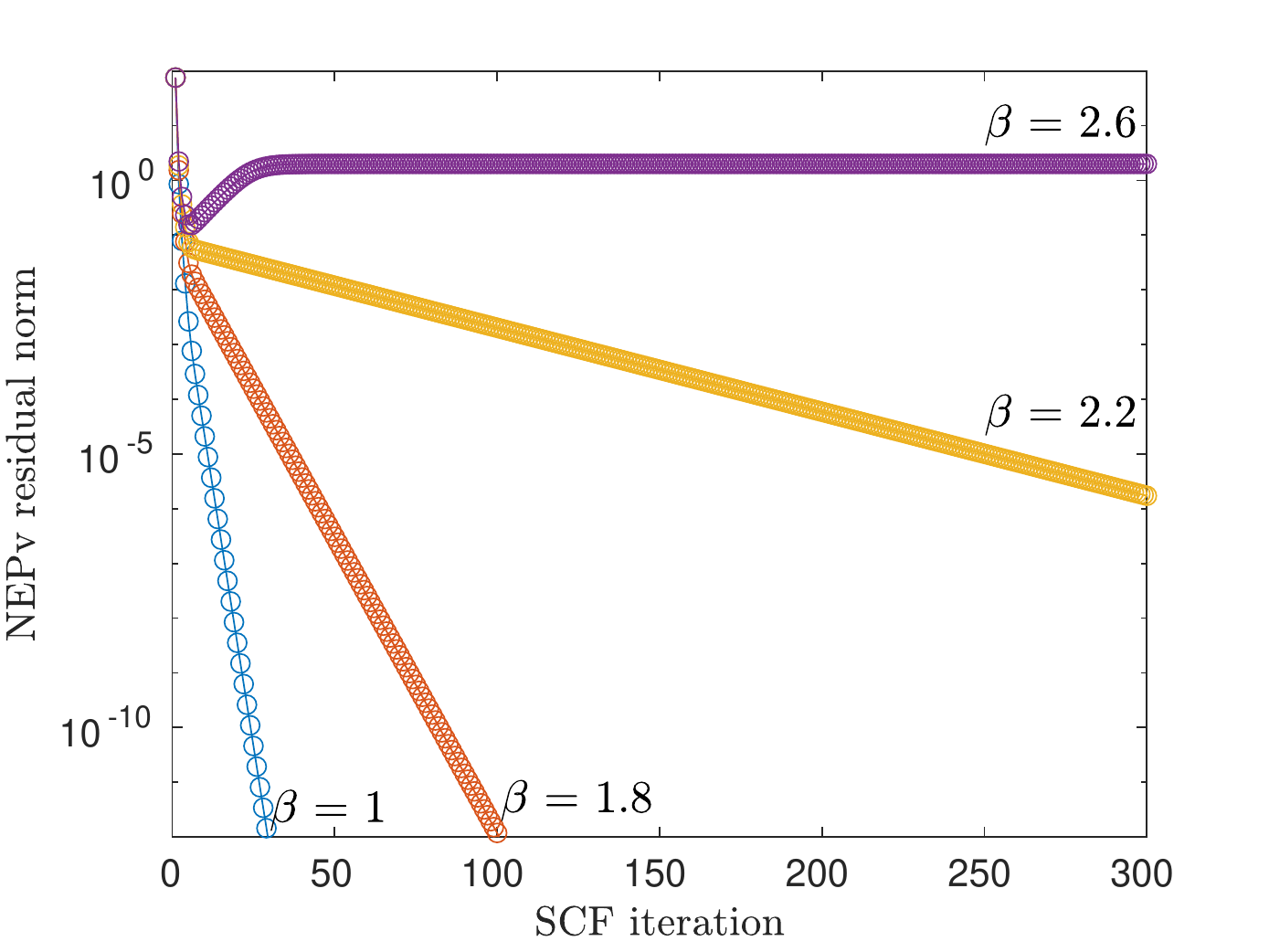}
\includegraphics[width=0.49\textwidth]{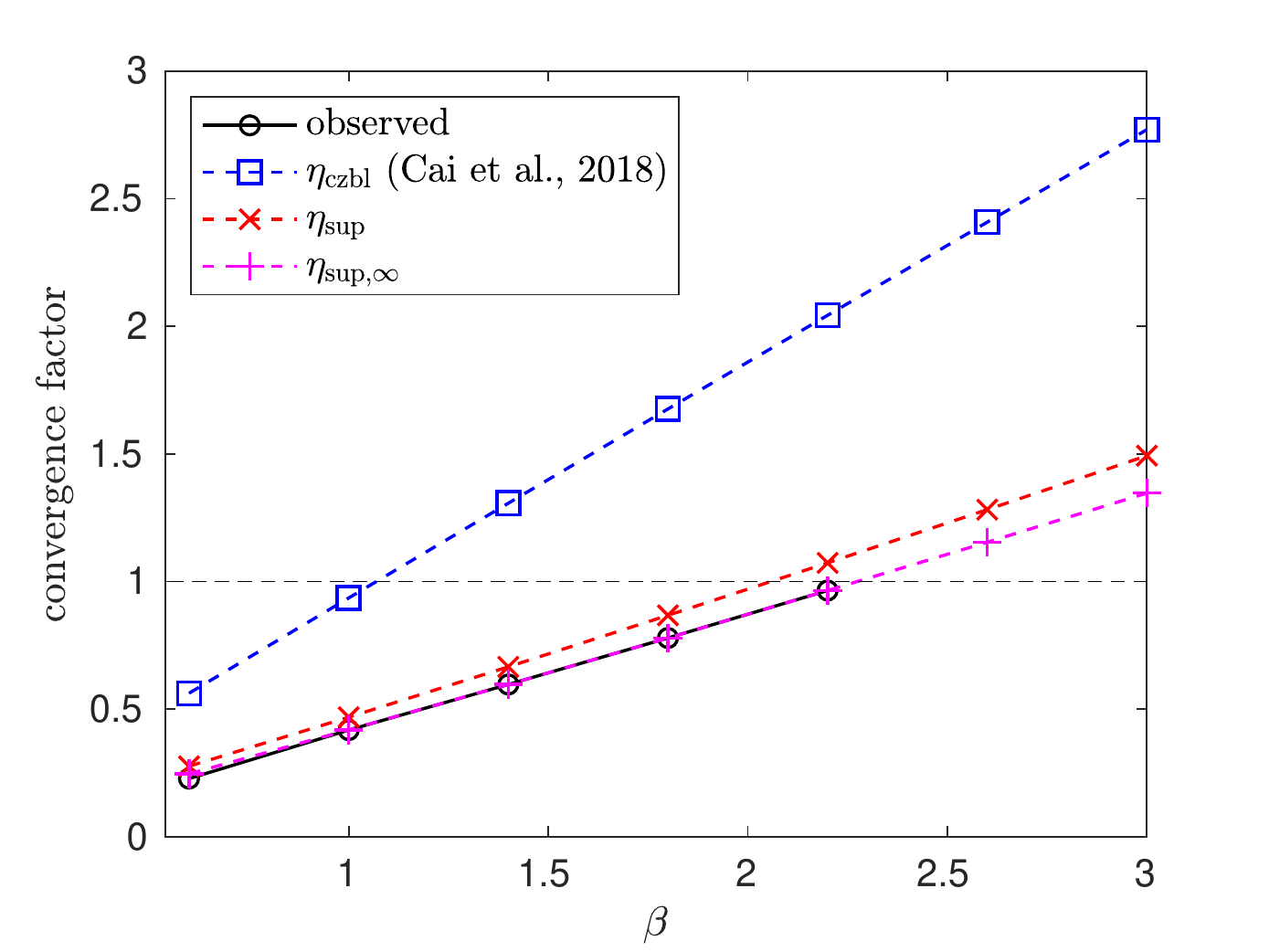}
\end{center}
\caption{Example~\ref{ex:gp2}: convergence history
of residual norm $\|H(V_i)V_i-V_i\Lambda_i\|_2$ by 
the plain SCF \eqref{eq:pscf} for selected $\beta$ (left); 
convergence rate estimates as  $\beta$ varies (right).
}\label{fig:ex:gp:2}
\end{figure}

\begin{figure}[ht]
\begin{center}
\includegraphics[width=0.49\textwidth]{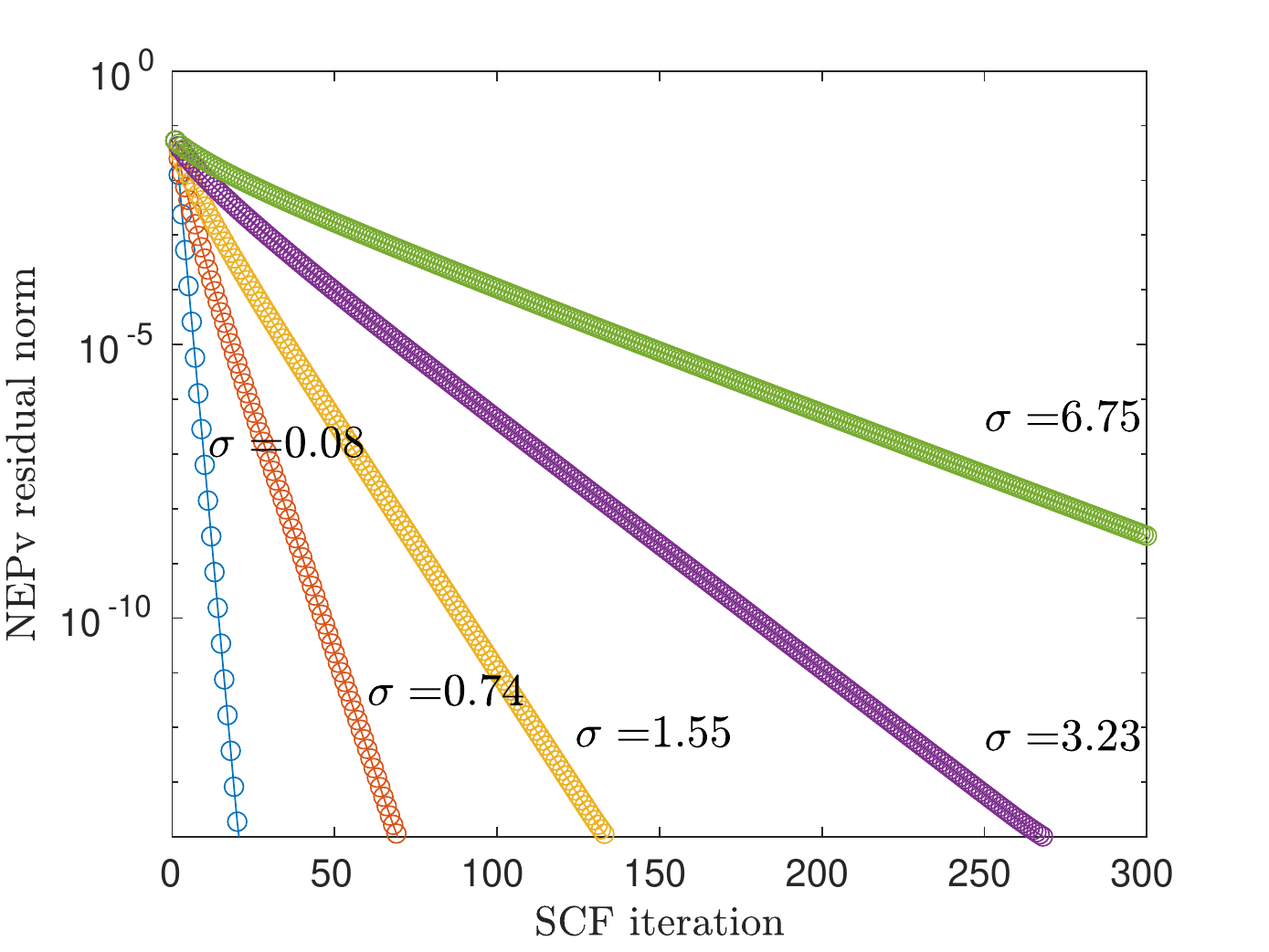}
\includegraphics[width=0.49\textwidth]{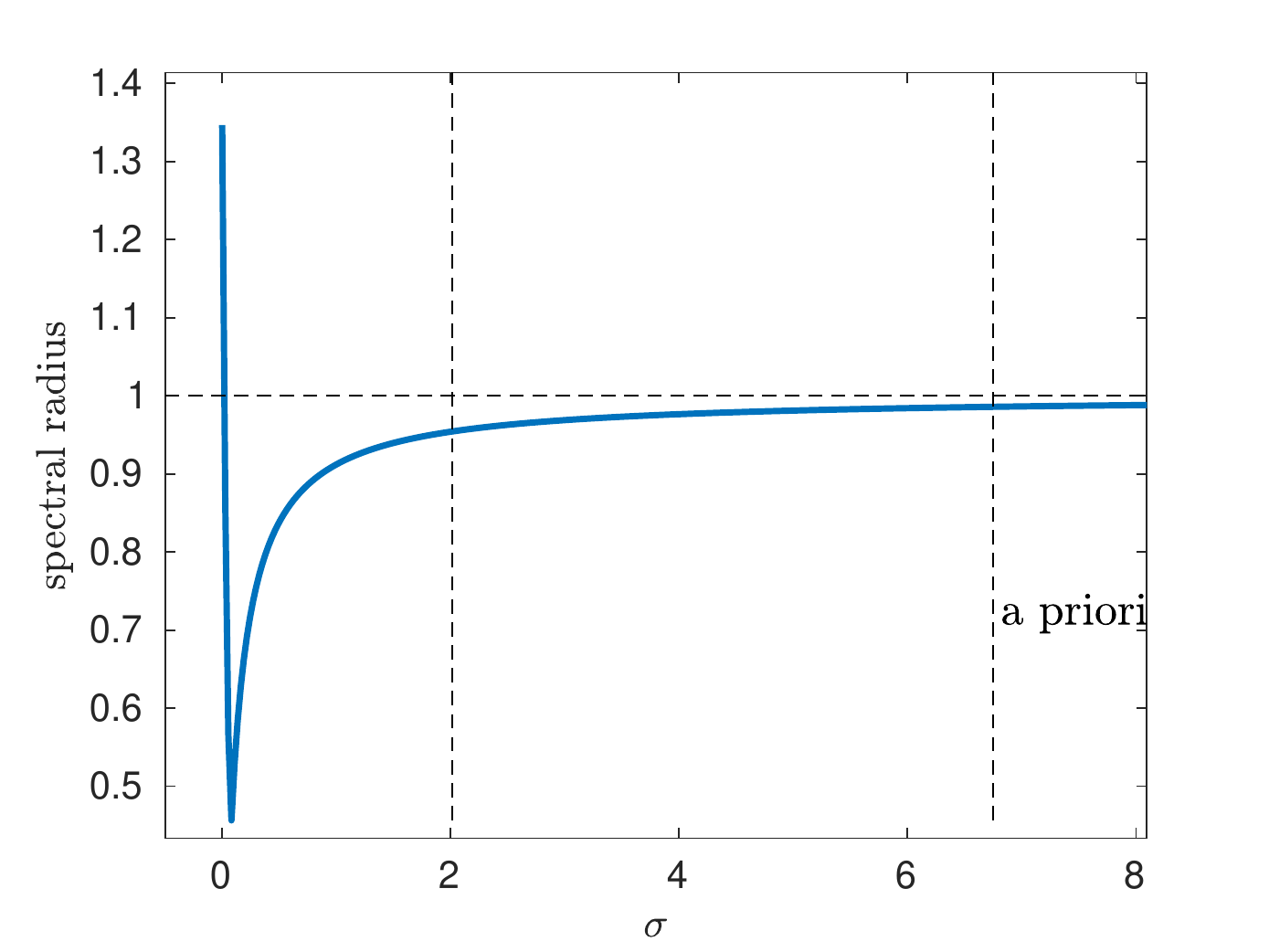}
\end{center}
\caption{Example~\ref{ex:gp2}:
convergence history of residual norm $\|H(V_i)V_i-V_i\Lambda_i\|_2$
by the level-shifted SCF \eqref{eq:levelshift} with selected $\sigma$ (left);
spectral radius of $\rho(\OP_{\sigma})$ as
shift $\sigma$ varies (right), where the first vertical dash line is 
$\sigma=\frac{\mu_{\max}}{2} -\delta_*$
suggested by~\eqref{eq:lsbnd} and the second is {\em a-priori} 
$\sigma =
\frac{1}{2}(3\beta+\|A_f\|_2)$ suggested by~\eqref{eq:sigmagp},
and the optimal shift is $\sigma\approx 0.08$.
The $H(V)$ is given by~\eqref{eq:gpe} with $\beta=3$.
}\label{fig:ex:gp2:rhos}
\end{figure}

\end{example}

\section{Concluding remarks}\label{sec:conclusion}
We have presented a comprehensive local convergence analysis 
of the plain SCF iteration and its level-shifted variant
for solving NEPv. The optimal convergence rate and
estimates are established.
Our analysis is in terms of the tangent-angle matrix to measure the
approximation error between consecutive SCF iterates  
and the intended target. 
We first established a
relation between the tangent-angle matrices associated 
with any two consecutive SCF approximates, and with it
we developed new formulas for the local error contraction factor and
the  asymptotic average contraction factor of SCF. 
The new formulas are sharper and
complement previously established local convergence results.
With the help of new convergence rate estimates, we
derive an explicit lower-bound on the shifting parameter 
to guarantee local convergence of the level-shifted SCF.
These results are numerically confirmed by examples from
applications in computational physics and chemistry.

Our analysis
does not cover more sophisticated variants of SCF 
such as the damped SCF~\cite{Cances:2000b}
and the Direct Inversion of 
Iterative Subspace (DIIS)~\cite{Pulay:1980,Pulay:1982}.
It is conceivable that by the tangent-angle matrix 
and the eigenspace perturbation theory, one can pursue 
the local convergence analysis of those variants. 

Finally, we note that we focused on NEPv \eqref{eq:nepv} 
satisfying the invariant property \eqref{eq:univar}. 
While this property is formulated as a result of some practically 
important applications,
there are recent emerging NEPv \eqref{eq:nepv} that do not have 
this property, such as the one in \cite{zhwb:2020}, and 
yet similar SCF iterations can be used. It would be interesting 
to find out what now determines the optimal
local convergence rate. This will be a future project to pursue. 


\appendix

\section{Adjoint operators} \label{app:adj}

The adjoint  $\OP^*$ in~\eqref{eq:lsreal} is derived as  follows.
\begin{align*}
\langle\, Y, \OP(Z)\,\rangle
    & = 2\alpha\,  \langle\, Y,\,\, D(V_*)\odot (V_{*\bot}^{\T}\, \Diag (L^{-1}\diag (V_{*\bot} Z V_*^{\T}))\,  V_* )\,\rangle\\
\text{\tiny(1)  by $ \langle Y, D\odot X\rangle  = \langle D\odot Y, X\rangle$\qquad}
    & = 2\alpha\,  \langle\, D(V_*)\odot Y,\,\, V_{*\bot}^{\T}\,
    \Diag (L^{-1}\diag (V_{*\bot} Z V_*^{\T}))\,  V_* \,\rangle\\
\text{\tiny (2) by $ \langle Y, AXB\rangle = \langle A^{\T}YB^{\T}, X\rangle $\qquad}
    & = 2\alpha\,  \langle\,  V_{*\bot}[ D(V_*)\odot Y]  V_*^{\T},\,\,
    \Diag (L^{-1}\diag (V_{*\bot} Z V_*^{\T})) \,\rangle\\
\text{\tiny (3) by $ \langle Y, \Diag  (b)\rangle = \langle \diag (Y),  b\rangle $\qquad}
    & = 2\alpha\,  \langle\,  \diag (V_{*\bot}[ D(V_*)\odot Y]  V_*^{\T}),\,\,
    L^{-1}\diag (V_{*\bot} Z V_*^{\T}) \,\rangle\\
\text{\tiny (4) by moving $L$ to the left\qquad}
    & = 2\alpha\,  \langle\,  L^{-1}\diag (V_{*\bot}[ D(V_*)\odot Y]  V_*^{\T}),\,\,
    \diag (V_{*\bot} Z V_*^{\T}) \,\rangle\\
\text{\tiny (5) by $ \langle \Diag  (b), Y\rangle = \langle b, \diag (Y)\rangle $\qquad}
    & = 2\alpha\,  \langle\,  \Diag (L^{-1}\diag (V_{*\bot}[D(V_*)\odot Y]  V_*^{\T})),\,\, V_{*\bot} Z V_*^{\T} \,\rangle.
\end{align*}
Finally, moving $V_{*\bot}$ and $V_*$ to the left
we obtain the formula~\eqref{eq:lsreal}.


The adjoint  $\OP^*$ in~\eqref{eq:lscmpx} is derived  analogously.
The first three steps are exactly the same as above, and so we continue with
\begin{align*}
    \langle\, Y, \OP(Z)\,\rangle
    & = 2\beta\,  \langle\, Y,\,\,
    D(V_*)\odot (V_{*\bot}^{\HH}\,
    \Diag ( \Re (\overline{V}_*\odot (V_{*\bot} Z)))\,  V_*)\,\rangle\\
\text{\tiny  by (1)--(3)   above \qquad}
    & = 2\beta\,  \langle\,  \diag (V_{*\bot}[ D(V_*)\odot Y]  V_*^{\HH}),\,\,
     \Re (\overline{V}_*\odot (V_{*\bot} Z)) \,\rangle\\
\text{\tiny  since it's vector inner product\qquad}
    & = 2\beta\,  \langle\,  \Re (\diag (V_{*\bot}[ D(V_*)\odot Y]
    V_*^{\HH})),\,\, \overline{V}_*\odot (V_{*\bot} Z) \,\rangle\\
\text{\tiny  by $ \langle Y, \Diag  (b)\rangle = \langle \diag (Y),  b\rangle $\qquad}
    & = 2\beta\,  \langle\,  \Re (\diag (V_{*\bot}[ D(V_*)\odot Y]
    V_*^{\HH}))\odot {V}_* ,\,\, V_{*\bot} Z \,\rangle.
\end{align*}
Finally, moving $V_{*\bot}$ to the left, we obtain the formula~\eqref{eq:lscmpx}.


\bibliographystyle{plain}
\bibliography{refs}

\end{document}